%% file: main.tex
\documentclass[final]{siamltex}
\usepackage{amsmath}
\usepackage{amssymb}
\usepackage{cases}
\usepackage{commath}
\usepackage{enumerate}
\usepackage{graphicx}
\usepackage{float}
\usepackage{hyperref}
\usepackage[symbol]{footmisc}
\usepackage{sidecap}

\usepackage{comment}

\input{mathdefs}
\input{projdefs}

\usepackage{color}

\title{Optimal Control of a  Free Boundary Problem: Analysis with Second Order Sufficient Conditions
}

\author{Harbir Antil\thanks{Department of Mathematical Sciences. George Mason University, Fairfax, VA 22030, USA ({\tt hantil@gmu.edu})}. 
\and Ricardo H. Nochetto\thanks{Department of Mathematics and Institute for Physical Science and Technology, University of Maryland
College Park, MD 20742, USA ({\tt rhn@math.umd.edu})}. 
\and Patrick Sodr{\'e}\thanks{Department of Mathematics, 
University of Maryland College Park, MD 20742, USA ({\tt sodre@math.umd.edu})}.}

\begin{document}
\maketitle
\begin{abstract}
We consider a PDE-constrained optimization problem governed 
by a free boundary problem. The state system is based on coupling the Laplace equation in the  bulk with a
Young-Laplace equation on the free boundary to account for surface tension, as proposed by P.\ Saavedra and L.\
R.\ Scott \cite{PSaavedra_RScott_1991}. This amounts to solving
a second order system both in the bulk and on the interface. Our analysis hinges on a convex 
control constraint such that the state constraints are always satisfied. Using only
first order regularity we show that the control to state operator is twice continuously Fr\'echet
differentiable. We improve slightly the regularity of the state variables and exploit it to show existence of a control together with second order sufficient optimality conditions. 
\end{abstract}

\input{01_Introduction/intro}

\input{02_Formulation/oc_fbp}

\input{03_AdHoc_Lagrange/oc_lagrangian}

\input{04_Control_to_State/g_map}
\input{04_Control_to_State/g_well_posed}

\input{04_Control_to_State/g_interpolation}
\input{04_Control_to_State/g_lipschitz}
\input{04_Control_to_State/g_differentiable}

\input{04_Control_to_State/g_diff_preliminaries}
\input{04_Control_to_State/g_first_order}
\input{04_Control_to_State/g_twice_diff}


\input{05_Reduced_Cost/oc_reduced_cost}
\input{05_Reduced_Cost/oc_rc_existence}
\input{05_Reduced_Cost/oc_rc_first}
\input{05_Reduced_Cost/adjoint_wellposed}

\input{05_Reduced_Cost/oc_rc_second}


\input{acknowledgements}

{\small
  \bibliographystyle{siam}
  \bibliography{references}
}

\end{document}

%% file: mathdefs.tex
\newcommand{\bbB}{\mathbb{B}}
\newcommand{\bbR}{\mathbb{R}}
\newcommand{\bbW}{\mathbb{W}}
\newcommand{\bbV}{\mathbb{V}}

\newcommand{\calB}{\mathcal{B}}
\newcommand{\calD}{\mathcal{D}}
\newcommand{\calJ}{\mathcal{J}}

\newcommand{\calU}{\mathcal{U}}

\newcommand{\aev}{\env{a.e. }}

\newcommand{\bdy}{\partial}

\newcommand{\cone}[1]{{\mathcal C\del{#1}}}

\newcommand{\CZinf}[1]{C_0^\infty\del{#1}}
\DeclareMathOperator{\divgText}{div}
\newcommand{\divg}{\divgText}
\newcommand{\difdir}[1]{\langle #1 \rangle}
\newcommand{\env}[1]{{\textnormal{#1}}}
\renewcommand{\eval}[1]{{#1}\bigg{|}}

\renewcommand{\grad}{\nabla}

\newcommand{\Lp}[1]{L^{p}\del{#1}}
\newcommand{\Lq}[1]{L^{q}\del{#1}}
\newcommand{\Ltwo}[1]{{L^2\del{#1}}}
\newcommand{\Linf}[1]{{L^\infty\del{#1}}}
\newcommand{\Lone}[1]{{L^1\del{#1}}}
\newcommand{\lapl}{\Delta}
\newcommand{\lagr}{\mathcal{L}}
\newcommand{\littleo}[1]{\mathop{o}\del{{#1}}}
\newcommand{\nd}{{\partial_{\nu}}}

\newcommand{\normLt}[2]{\norm{#1}_\Ltwo{#2}}
\newcommand{\normLi}[2]{\norm{#1}_\Linf{#2}}

\newcommand{\normS}[4]{\norm{#1}_\sob{#2}{#3}{#4}}
\newcommand{\normSD}[4]{\norm{#1}_{\sob{#2}{#3}{#4}^*}}
\newcommand{\normSZ}[4]{\abs{#1}_\sob{#2}{#3}{#4}}
\newcommand{\normSZD}[4]{\norm{#1}_{\sobZ{#2}{#3}{#4}^*}}
\newcommand{\of}[1]{\del{#1}}
\newcommand{\pair}[1]{\left\langle #1 \right\rangle}

\newcommand{\pairLt}[2]{{\mathop{\pair{#1}}}_{\Ltwo{#2}\times\Ltwo{#2}}}
\newcommand{\pairLpq}[2]{{\mathop{\pair{#1}}}_{\Lp{#2}\times\Lq{#2}}}

\newcommand{\pairSob}[5]{{\mathop{\pair{#1}}}_{\sob{-#3}{#2}{#5}\times\sobZ{#3}{#4}{#5}}}
\renewcommand{\restriction}{\vert}
\newcommand{\Res}{\mathfrak{R}}
\newcommand{\seq}[1]{\cbr{#1}}
\newcommand{\setdef}[2]{\cbr{#1\,:\, #2}}
 \newcommand{\sgn}{\operatorname{sgn}}

\newcommand{\sob}[3]{{W^{#1}_{#2}\del{#3}}}
\newcommand{\sobZ}[3]{{\mathring{W}^{#1}_{#2}\del{#3}}}

\newcommand{\totalD}{\operatorname{d}\!}

\newcommand{\weakto}{\rightharpoonup}

\newtheorem{thm}[theorem]{Theorem}
\newtheorem{lem}[theorem]{Lemma}
\newtheorem{prop}[theorem]{Proposition}
\newtheorem{cor}[theorem]{Corollary}

\newtheorem{defn}[theorem]{Definition}

\newcommand{\I}{\text{\rm I}}

%% file: projdefs.tex
\newcommand{\DO}{\calD_\Omega}

\newcommand{\BO}{\calB_\Omega}
\newcommand{\BG}{\calB_\Gamma}

\newcommand{\costF}{\mathcal{J}\del{\gamma, y, u}}

\newcommand{\optimal}{\del{\gop,\yop,\uop,\rop,\sop}}
\newcommand{\Omegag}{{\Omega_\gamma}}

\renewcommand{\restriction}{\vert}
\newcommand{\yop}{\bar{y}}
\newcommand{\uop}{\bar{u}}

\newcommand{\gop}{\bar{\gamma}}
\newcommand{\ghat}{\widehat \gamma}
\newcommand{\gtl}{\widetilde{\gamma}}
\newcommand{\yhat}{\widehat y}
\newcommand{\ytl}{\widetilde{y}}
\newcommand{\rop}{\bar{r}}
\newcommand{\rtl}{\widetilde{r}}
\newcommand{\sop}{\bar{s}}
\newcommand{\stl}{\widetilde{s}}

\newcommand{\Uad}{{\mathcal{U}_{ad}}}


%% file: 01_Introduction/intro.tex
\section{Introduction} \label{s:intro}
Free boundary problems (FBPs) are challenging due to their highly nonlinear nature. Besides the 
state variables, the domain is also an unknown. FBPs find a wide range of applications from phase
separation (Stefan problem, Cahn-Hilliard), shape optimization (minimal surface area), optimal 
control problems with state constraints, fluid dynamics (flow in porous media), crystal growth, 
biomembranes, electrowetting on dielectric, to finance. For many of these problems  there is a close 
interplay between the surface tension and the curvature of
the interface \cite{SWalker_BShapiro_RNochetto_2009a, SWalker_ABonito_RNochetto_2010a}. 
\begin{figure}[H]
\centering
\includegraphics[width=0.3\textwidth,height=0.3\textwidth]{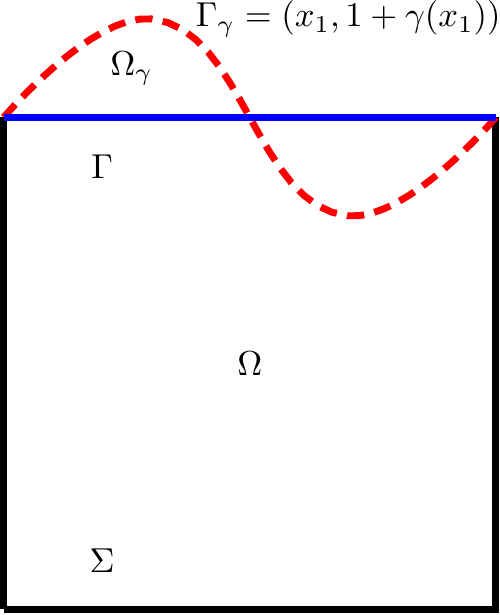}
\caption{$\Omega_{\gamma}$ denotes a physical domain with boundary 
               $\partial \Omega_{\gamma} = \Sigma \cup \Gamma_{\gamma}$. Here $\Sigma$ includes the lateral and the 
               bottom boundary and is assumed to be fixed. Furthermore, the top boundary $\Gamma_{\gamma}$ (dotted line) 
               is ``free" and is assumed to be a graph of the form $(x_1,1+\gamma(x_1))$, where $\gamma \in \sobZ{1}{\infty}{0,1}$ 
               denotes a parametrization. $\Gamma_{\gamma}$ is further mapped to a fixed boundary $\Gamma = (0,1) \times \{1\}$ and in 
               turn the physical domain $\Omega_{\gamma}$ is mapped to a reference domain $\Omega = (0,1)^2$, where all 
               computations are carried out.} 
\label{f:domain}
\end{figure}
\noindent
Of particular interest to us is the control of a model FBP previously studied by 
P. Saavedra and L. R. Scott in \cite{PSaavedra_RScott_1991} and formulated in graph form; see Figure~\ref{f:domain} where the free boundary $\Gamma_\gamma$ is the dotted line.  
The state equations \eqref{eq:state_eqn} involve 
a Laplace equation in the  bulk and a Young-Laplace equation on the free boundary to account for surface tension. 
This amounts to solving a second-order system both in the bulk and on the interface. Below we give a detailed 
description of the problem. 

Let $\sobZ{1}{\infty}{\I}$ be the Sobolev space of Lipschitz
continuous functions on the unit interval $\I=(0,1)$ which vanish at $0$ and $1$. 
Let $\gamma \in \sobZ{1}{\infty}{\I}$ denote a parametrization of the top boundary (see \figref{f:domain}) of 
the physical domain 
$\Omega_{\gamma} \subset \Omega^* \subset \mathbb{R}^2$ with boundary 
$\partial \Omega_{\gamma} := \Gamma_{\gamma} \cup \Sigma$, defined as 
\begin{subequations}
	\begin{align*}
 		\Omega^* &= (0,1) \times (0,2), \\
 		\Omegag &= \setdef{(x_1,x_2)}{x_1 \in \I, \;0< x_2 < 1+ \gamma(x_1)}, \\
 		\Gamma_{\gamma} &= \setdef{(x_1,x_2)}{x_1\in\I,\; x_2=1+\gamma(x_1)}, \\ 
 		\Sigma &= \bdy\Omega_{\gamma} \setminus \Gamma_{\gamma}, \\
		 \Gamma &= \setdef{(x_1,x_2)}{ x_1 \in \I,\; x_2=1} ;
	\end{align*}
\end{subequations}
$\Omega^*$ and $\Sigma$ are fixed while $\Omega_{\gamma}$ and $\Gamma_{\gamma}$ deform
according to $\gamma$. 
Hereafter we will identify $\Gamma$ with
 $\I$ as well as Sobolev spaces defined on them. Since $\gamma$ is Lipschitz continuous with constant $1$, according to 
\eqref{eq:state_const_a}, we deduce that $\abs{\gamma} \leq 1/2$; this guarantees that $\Omega_\gamma \subset \Omega^*$.

We want to
find an optimal control $u \in \Uad \subset \Ltwo{\I}$ so that the solution pair $\del{\gamma,y}$ of the FBP
approximates a given boundary $\gamma_d: \I \to \mathbb{R}$ and 
potential $y_d: \Omega^* \to \mathbb{R}$. This amounts to solving the problem: minimize
\begin{subequations}\label{eq:ocfbp}
	\begin{equation}\label{eq:cost_fcn}
	     \costF := \frac{1}{2}\normLt{\gamma - \gamma_d}{\I}^2  
				+ \frac{1}{2}\normLt{y - y_d}{\Omegag}^2 
				+ \frac{\lambda}{2}\normLt{u}{\I}^2,
	\end{equation}
	subject to the state equations 
	\begin{equation}\label{eq:state_eqn}\left\{
		\begin{aligned} 
			-\lapl y &= 0 && \env{in }\Omegag \\ 
	    		y &= v  && \env{on } \bdy\Omegag \\
	    	-\kappa \mathcal H\sbr{\gamma} 
	    	+ \nd y\del{\cdot,1+\gamma} &=  u
	    		&& \env{on } \I \\
	    	\gamma(0) = \gamma(1) &= 0,
		\end{aligned}
	\right.\end{equation}
	the state constraints
	\begin{align}\label{eq:state_const_a}
		&\abs{\totalD_{x_1}\gamma} \leq 1 \quad \aev \I,
	\end{align}
	with $\totalD_{x_1}$ being the total derivative with respect to $x_1$, 
	and the control constraint 
	\begin{align}\label{eq:control_const_a}
		u &\in \Uad
	\end{align}
	dictated by $\Uad$, a closed ball in $\Ltwo{\I}$, to be specified later in Definition~\ref{defn:const_set}.
\end{subequations}
Here $\lambda > 0$ is a stabilization parameter; 
$v \in \sob1p{\Omega^*}$, $p>2$ \cite[Lemma 2]{PSaavedra_RScott_1991}, 
is given which in principle 
could act as a Dirichlet boundary control;
$$
\mathcal H\sbr{\gamma} := \totalD_{x_1}\del{\frac{\totalD_{x_1}\gamma}{\sqrt{1+\abs{\totalD_{x_1}\gamma}^2}}}
$$ 
is the \emph{curvature} of $\gamma$; and $\kappa > 0$ plays the role of surface tension coefficient.

Optimal control of partial differential equations (PDEs) allows us to achieve a specific goal \eqref{eq:cost_fcn} with 
PDE \eqref{eq:state_eqn} and other constraints \eqref{eq:state_const_a}-\eqref{eq:control_const_a} being satisfied 
and can be highly beneficial in practice (see \cite{FTroltzsch_2010a} for more details). 
For example using the reverse electrowetting, i.e. by applying a control to change the shape of fluid droplets,
one can generate enough power to charge a cellphone \cite{TKrupenkin_JATaylor_2011a}.
There has been various attempts to solve optimal control problems with a FBP constraint. 
We refer to \cite{MHinze_SZiegenbalg_2007b, MHinze_SZiegenbalg_2007a} for control
of a two phase Stefan problem in graph formulation and \cite{MKBernauer_RHerzog} for the same problem
in level set formulation. Paper \cite{SRepke_NMarheineke_RPinnau2010a} discusses optimal control
of a FBP with Stokes flow. 
Even though problem \eqref{eq:cost_fcn}-\eqref{eq:control_const_a} is relatively simple, it captures the essential features associated with surface tension effects found in more complex systems, and allows us to develop a complete second-order analysis, 
based on \cite{FTroltzsch_2010a}, which is absent in the existing literature on FBP.

Depending on the role of the free boundary there are several methodologies to formulate a FBP. We choose the sharp interface method written in graph form (see Figure~\ref{f:domain}). The (free) interface $\Gamma_\gamma$ is governed by the explicit nonlinear PDE
	\[
		-\kappa \mathcal H\sbr{\gamma} + \nd y =  u. 
	\]
A similar approach was used in \cite{MHinze_SZiegenbalg_2007b, MHinze_SZiegenbalg_2007a}
for the optimal control of a Stefan problem, but without the full accompanying theory developed herein. Alternative approaches to treat FBPs are the level set method and the diffuse
interface method \cite{KDeckelnick_GDziuk_CElliott_2005a,MKBernauer_RHerzog}.

We use a fixed domain approach to solve the optimal control free boundary problem (OC-FBP). In fact, we transform $\Omega_\gamma$ to $\Omega=(0,1)^2$ and $\Gamma_\gamma$ to $\Gamma=(0,1)\times\set{1}$ (see Figure~\ref{f:domain}), at the expense of having a governing PDE with rough coefficients. This avoids dealing with shape sensitivity analysis \cite{JSokolowski_JPZolesio_1992,
MDelfour_JZolesio2011}. We refer to \cite{KVZee_EBrummelen_IAkkerman_RBorst_2010a} for a comparison 
between these approaches applied to a FBP. Using operator interpolation \cite{LTartar_2007} we
demonstrate how to improve the existing regularity of state variables derived earlier in
\cite{PSaavedra_RScott_1991}, which turns out to be instrumental to derive the second-order sufficient condition. 

One of the challenges of an OC-FBP is dealing with possible
topological changes of the domain by introducing state constraints. Our analysis provides control constraints which always enforce the state constraints, i.e. we can simply
treat OC-FBP as a control constrained problem without any state constraints. We refer to \cite[Section~6.2]{FTroltzsch_2010a} and references therein for state and gradient constraints problems along with the associated difficulties.
We will provide a comprehensive numerical approach for the control problem \eqref{eq:ocfbp} in a forthcoming paper
\cite{HAntil_RHNochetto_PSodre_2013a}.

We have organized this 
paper 
as follows. A detailed problem description on a fixed domain is given in \autoref{s:oc_fbp}.
We introduce the Lagrangian functional to formally derive the first-order necessary optimality conditions in
\autoref{s:oc_lagrangian}. We present a rigorous justification of the Lagrangian results in
the remaining sections. To this end, we introduce a control-to-state operator in \autoref{s:g_map} and show that for a particular set of admissible controls it is twice Fr\'echet
differentiable. Finally, we write the optimal control problem in its reduced form and show the existence of a control under slightly higher regularity together with second-order sufficient conditions in \autoref{s:oc_rc}.

%% file: 02_Formulation/oc_fbp.tex
\section{OC-FBP on Reference Domain}
\label{s:oc_fbp}

We start by mapping the physical domain $\Omega_{\gamma}$ onto the fixed reference domain 
$\Omega=\intoo{0,1}^2$. 
This results in an optimal control problem subject to PDE constraints with nonlinear coefficients depending on $\gamma$ but without an explicit interface. 
The idea is to map the unknown domain $\Omegag$ onto the fixed domain
$\Omega$ using the inverse of the Lipschitz map $\Psi :
\Omega \to \Omegag$ defined as
	\begin{align}\label{eq:ref_dom_map}
		\Psi(x_1,x_2) = \del{x_1, \del{1+\gamma\del{x_1}} x_2} \in
		\Omegag, \quad \env{for } \del{x_1,x_2} \in \Omega.
	\end{align}							
Since $\gamma$ is Lipschitz continuous with constant $1$, according to 
\eqref{eq:state_const_a}, we deduce that $\abs{\gamma} \leq 1/2$ and
that $\Psi$ is invertible because its Jacobian is $J_\gamma=1+\gamma$. Furthermore, the inverse of $\Psi$ is also Lipschitz. Moreover, it becomes routine to check that the Laplace equation $\Delta y = 0$ in $\Omegag$ and $\nd y$ on $\Gamma_\gamma$ can be written as
$$
   \divg\of{A\sbr{\gamma}\grad y} = 0 \quad \mbox{in } \Omega, \qquad
    A\sbr{\gamma}\grad y \cdot \nu \del{1+\abs{\totalD_{x_1}\gamma}^2}^{-1/2} \quad \mbox{on } \I,
$$
where $\nu = \sbr{0,1}^{T}$, and $A : \sobZ{1}{\infty}{\I} \rightarrow 
L^\infty(\Omega)^{2\times 2}$ is the Nemytskii operator \cite[Chapter 4]{FTroltzsch_2010a} defined by
\begin{align} \label{eq:op_A}
	A\sbr{\gamma} &
				  = \begin{bmatrix} 1+\gamma(x_1) & -\totalD_{x_1}\gamma(x_1)x_2 \\ 				
				  			  -\totalD_{x_1}\gamma(x_1)x_2 & \frac{1+\del{\totalD_{x_1}\gamma(x_1)x_2}^2}{1+\gamma(x_1)} \end{bmatrix} .
\end{align} 
It is convenient to write $A_{2,2}\sbr{\gamma}$ as $\Phi\sbr{\gamma} := \varphi\del{\gamma(x_1), \totalD_{x_1}\gamma(x_1)x_2}$, where $\varphi\del{a, b} := \del{1 + b^2}/\del{1 + a}.$
 
We simplify the exposition by exploiting the fact that we are only interested in small $\sobZ1\infty{\I}$-perturbations 
of the flat case $\gamma=0$, namely we have
$\totalD_{x_1}\gamma$ small pointwise. We thus make the following assumptions:
\begin{itemize}
 \item[{$\bf (A_1)$}] Linearized curvature: $\mathcal{H}_{lin}\sbr{\gamma} = {\totalD^2_{x_1}\gamma} \del{1+\abs{\totalD_{x_1}\gamma}^2}^{-1/2}$.
 \item[{$\bf (A_2)$}] Scaled control: $u$ becomes $u \del{1+\abs{\totalD_{x_1}\gamma}^2}^{-1/2}$. 
\end{itemize}
These assumptions are not crucial. 
The nonlinear curvature $\mathcal{H}[\gamma]$ formally reads
\[
    \mathcal{H}[\gamma] = \totalD_{x_1}\del{\frac{\totalD_{x_1}\gamma}{\sqrt{1+\abs{\totalD_{x_1}\gamma}^2}}}  
        = \frac{\totalD_{x_1}^2 \gamma}{ \del{ 1+\abs{\totalD_{x_1}\gamma}^2}^{3/2}} ,
\]
which is similar to the linearized curvature $\mathcal{H}_{lin}$ except for the $L^\infty(\I)$ 
factor $1/(1+\abs{\totalD_{x_1} \gamma}^2) \approx 1$. Assumption ${\bf (A_1)}$ simplifies the structure of the bilinear form $\BG$ in \eqref{eq:g_var_pde}. On the other hand, the scaling of the control in ${\bf (A_2)}$ avoids unnecessarily complicating the right-hand-side of \eqref{eq:lin_st_const} below, which would contain $u \del{1+\abs{\totalD_{x_1}\gamma}^2}^{1/2}$ instead of simply $u$. Our analysis below extends to the general setting without assumptions ${\bf (A_1)}$ and ${\bf (A_2)}$.

Under these assumptions and applying the map $\Psi$, the optimal control problem 
\eqref{eq:ocfbp} becomes: minimize
\begin{subequations}
	\begin{equation}\label{eq:lin_cost}
		\costF := \frac{1}{2}\normLt{\gamma - \gamma_d}{\I}^2 
					+ 
                                          \frac{1}{2}\normLt{\del{y +
                                              v - y_d}
                                            \sqrt{J_\gamma}}{\Omega}^2 
					+ \frac{\lambda}{2}\normLt{u}{\I}^2
	\end{equation}
subject to the state equations $\del{\gamma,y}\in\sobZ 1\infty{\I} \times \sobZ 1 p \Omega$
	\begin{equation}\label{eq:lin_st_const} \left\{
		\begin{aligned} 
			-\divg\of{A\sbr{\gamma}\grad\del{y+v}} &= 0 				&& \env{in } \Omega \\ 	        	                       
	        -\kappa \totalD_{x_1}^2\gamma + A\sbr{\gamma}\grad\del{y+v}\cdot\nu &= u 	&& \env{in } \I \\		    					  
		\end{aligned}
	\right.\end{equation}	
	the state constraints
	\begin{align}\label{eq:state_const}
		&\abs{\totalD_{x_1}\gamma} \leq 1 \quad \aev  \I,
	\end{align}
	with $\totalD_{x_1}$ being the total derivative with respect to $x_1$, 
	and the control constraint 
	\begin{align}\label{eq:control_const}
		u &\in \Uad
	\end{align}
	dictated by $\Uad$, a closed ball in $\Ltwo{I}$, to be specified later in Definition~\ref{defn:const_set}.
\end{subequations}

In order to derive the first- and second-order optimality conditions in later sections, we need to compute the first- and 
second-order directional derivatives of $A$, which in turn requires computing the directional derivative of the 
Nemytskii operator $\Phi$ defined above. To simplify notation, we drop the evaluation
of $\gamma$ and $\totalD_{x_1}\gamma$ at $x_1$. The derivative of $\Phi$ in the direction $h$ at 
$(\gamma, x_2 \totalD_{x_1}\gamma)$ is given by
\begin{align*}
     \Dif\Phi\sbr{\gamma}\difdir{h} 
                 =  \partial_a\varphi\del{\gamma, x_2 \totalD_{x_1}\gamma} h + \partial_b\varphi\del{\gamma, x_2 \totalD_{x_1}\gamma}  x_2 \totalD_{x_1}h
\end{align*}
where 
\begin{align*}
	\grad \varphi(a,b) 
	    &:= \begin{bmatrix} \partial_a\varphi(a, b) & \partial_b\varphi(a,b)\end{bmatrix} 
	       = \begin{bmatrix} -\frac{1 + b^2}{(1+ a)^2} &	\frac{2 b}{1+a}
	          \end{bmatrix}    .
\end{align*}
Furthermore, we obtain the following representation for $\Dif A$ in terms of $h$ and $\totalD_{x_1} h$
\begin{align} \label{eq:op_DA}
     \Dif A\sbr{\gamma}\difdir{h} 
       & := A_1\sbr{\gamma}h + A_2\sbr{\gamma} \totalD_{x_1}h  \nonumber \\
       & ~=  \begin{bmatrix} 1 & 0 \\ 0 & \partial_a\varphi\del{\gamma, x_2 \totalD_{x_1}\gamma} \end{bmatrix} h                                         
          + \begin{bmatrix} 0 & -x_2 \\ -x_2 & x_2 \partial_b\varphi\del{\gamma, x_2 \totalD_{x_1}\gamma} \end{bmatrix}  \totalD_{x_1}h , 
\end{align}
whence the remainder $\Res_A\sbr{\gamma, h}$ at $\gamma$ in the direction $h$ reads
\begin{subequations}\label{eq:op_RA}
	\begin{align}
		\Res_A\sbr{\gamma, h} := A\sbr{\gamma+h} - A\sbr{\gamma} - \Dif A\sbr{\gamma}\difdir{h}
	\end{align}
and
	\begin{align}
		\lim_{\normSZ{h}1\infty{\I} \to 0} 
			\frac{\normLi{\Res_A\sbr{\gamma, h}}\Omega}{\normSZ{h}1\infty{\I}} = 0 ;
	\end{align}
	(\ref{eq:op_RA}b) follows directly from the structure of $A$
        \cite[Lemma 4.12]{FTroltzsch_2010a}.  
\end{subequations}
The Hessian of $\varphi$ is
\begin{align*}
    \nabla^2\varphi(a, b)
      &= \begin{bmatrix} 
              \partial_{a}^2\varphi(a, b) & \partial_{ab}\varphi(a,b) \\ \partial_{ba}\varphi(a,b) & 
             \partial_{b}^2\varphi(a, b) 
            \end{bmatrix}
      = 2\begin{bmatrix} 
               \frac{1 + b^2}{(1+a)^3} & \frac{-b}{(1+a)^2} \\ 
               \frac{-b}{(1+a)^2} & \frac{1}{1+a}
            \end{bmatrix}. 
\end{align*}
The second-order derivative of $\Phi$ in the direction $h_1$ followed by $h_2$ 
evaluated at $(\gamma, x_2 \totalD_{x_1}\gamma)$ is
\begin{align*}
     \Dif^2\Phi\sbr{\gamma}\difdir{h_2,h_1} 			
	&= \partial_{a}^2\varphi\del{\gamma, x_2  \totalD_{x_1}\gamma } h_2 h_1 
	               + \partial_{ab}\varphi\del{\gamma, x_2 \totalD_{x_1}\gamma } x_2 h_2 \totalD_{x_1}h_1  \\
	 &\quad + \partial_{ab}\varphi\del{\gamma, x_2 \totalD_{x_1}\gamma } x_2 \totalD_{x_1}h_2 h_1  
         	     + \partial_{b}^2\varphi\del{\gamma, x_2 \totalD_{x_1}\gamma } x_2^2 \totalD_{x_1}h_2 \totalD_{x_1}h_1.
\end{align*}
Finally, we obtain the following representation for $\Dif^2 \!\! A$ in terms of $h_1$ and $h_2$
\begin{align} \label{eq:op_D2A}
     \Dif^2 \!\! A\sbr{\gamma}\difdir{h_2,h_1} = \begin{bmatrix} 0 & 0 \\ 0 & \Dif^2\Phi\sbr{\gamma}\difdir{h_2,h_1}\end{bmatrix},
\end{align}
whence the remainder $\Res_{\Dif A}\sbr{\gamma, h_1, h_2}$ at $\gamma$ reads
	\begin{subequations}\label{eq:op_RDA}
	\begin{equation}
		\Res_{\Dif A}\sbr{\gamma, h_1, h_2} 
			:= \Dif A\sbr{\gamma+h_2}\difdir{h_1} - \Dif A\sbr{\gamma}\difdir{h_1} - \Dif^2 \!\! A\sbr{\gamma}\difdir{h_1, h_2}, 
	\end{equation}
and \cite[Lemma 4.12]{FTroltzsch_2010a}
	\begin{equation}
		\lim_{\substack{\normSZ{h_1}1\infty{\I} \to 0 \\ \normSZ{h_2}1\infty{\I} \to 0 }}
			\frac{\norm{\Res_{\Dif A}\sbr{\gamma, h_1, h_2}}_{L^\infty(\Omega)^{2 \times 2}}}{\normSZ{h_1}1\infty{\I}\normSZ{h_2}1\infty{\I}} = 0. 
	\end{equation}
	\end{subequations}
\begin{prop}[bounds on $A$] \label{prop:op_A}
If the state constraint \eqref{eq:state_const} holds, then there exists a positive constant $C_A < \infty$ such that
\begin{equation}\begin{aligned}  \label{eq:A_unif_bound}
    \norm{A\sbr\gamma}_{L^\infty\del{\Omega}^{2\times 2}}
    	&+ \sup_{\normSZ{h}1\infty{\I}=1}\norm{\Dif A\sbr\gamma\difdir{h}}_{L^\infty\del{\Omega}^{2\times 2}} \\	&+ \sup_{\substack{\normSZ{h_1}1\infty{\I} = 1 \\ \normSZ{h_2}1\infty{\I} = 1}}\norm{\Dif^2 \!\! A\sbr{\gamma}\difdir{h_1,h_2}}_{L^\infty\del{\Omega}^{2\times 2}} \leq C_A.
\end{aligned}\end{equation}    
\end{prop}

%% file: 03_AdHoc_Lagrange/oc_lagrangian.tex
\section{Formal Lagrangian Formulation}
\label{s:oc_lagrangian}

In this section we formally derive  
the first-order necessary optimality conditions using the Lagrangian approach described in
\cite{FTroltzsch_2010a}. To this end, we will assume that the admissible control set $\Uad$ 
guarantees the state constraints \eqref{eq:state_const}, a pending issue we revisit and 
examine in detail in section \ref{s:g_map}.
For a rigorous analysis of the existence of Lagrange multipliers in Banach spaces we 
refer to \cite{JZowe_SKurcyusz_1979a}.
 
It is well known that for a convex optimal control problem with linear
constraints, the first-order necessary optimality conditions are also sufficient conditions 
\cite[Lemma 2.21]{FTroltzsch_2010a}. This does not apply to our problem because, despite 
linearizing the curvature via assumption {$(\bf A_1)$}, the state equations \eqref{eq:lin_st_const} are still highly nonlinear and the optimization nonconvex. We will  derive the second-order sufficient optimality conditions in  \autoref{s:oc_rc}.

Let $s, r$ denote the adjoint variables corresponding to states $\gamma, y$ respectively. 
Then the formal Lagrangian functional is given by
\begin{equation}\begin{aligned}\label{eq:lin_lagr}
	\lagr (\gamma,y,u,r,s) 
		:= \costF  &+\int_\Omega \divg\of{A\sbr{\gamma}\grad \del{y+v}}r \dif x \\
		&+ \int_0^1 \del{\kappa \totalD^2_{x_1}\gamma - A\sbr\gamma\grad\del{y+v}\cdot\nu + u}s \dif\sigma ;
\end{aligned}\end{equation}
we implicitly assume regularity in writing \eqref{eq:lin_lagr}. 
Additionally, if $\optimal$ is a critical point for $\lagr$, then the first-order
necessary optimality conditions are
\begin{subequations} \label{eq:first_order_opt}
	\begin{align}
		\pair{\Dif_{v}\lagr{\optimal},h}_{\mathbb{V}^*,\mathbb{V}}   &= 0   
		\qquad \forall h \in \mathbb{V},  \\
		\pairLt{\Dif_{u}\lagr{\optimal},u-\uop}{\I}   &\ge 0 
                \qquad \forall u \in \mathcal{U}_{ad},
	\end{align}
where 
\begin{align}\label{eq:first_order_opt_spaces}
    \del{v,\mathbb{V}} \in \set{ \del{\gamma,\sobZ1\infty{\I}}, \del{y,\sobZ1p\Omega}, \del{r,\sobZ{1}{q}{\Omega}}, \del{s,\sobZ11{\I}} } ,
\end{align}
\end{subequations}
$p,q$ are H\"older conjugate indices, i.e. $1/p+1/q=1$, with $p>2$, and $\pair{\cdot,\cdot}_{\mathbb{V}^*,\mathbb{V}}$ stands for duality pairings. 
Therefore, computing $\optimal$ requires solving the nonlinear system
\eqref{eq:first_order_opt}. We again point out that the calculations in this section are merely formal and the functions in \eqref{eq:first_order_opt_spaces} will be justified later in sections \ref{s:g_map} and \ref{s:oc_rc}. In practice this can be realized using techniques described in 
\cite{HAntil_RHoppe_CLinsenmann_2007a, CTKelley_1999a, FTroltzsch_2010a}.
For variational inequalities of the first kind, such as (\ref{eq:first_order_opt}b), we refer to 
\cite{RGlowinski_1984a} for relaxation and augmented Lagrangian 
techniques and to \cite{JCDReyes_MHintermuller_2011a} for semi-smooth Newton methods.
The remainder of this section is devoted to the derivation of the equations satisfied by $\optimal$ using the nonlinear system above.

Since $\pair{\Dif_{\cbr{s,r}}\lagr{\optimal},h} = 0$ implies that $(\gop,\yop)$ solves the state equations \eqref{eq:lin_st_const}, we focus on the adjoint equations $\pair{\Dif_{\cbr{\gamma,y}}\lagr{\optimal},h}=0$.  Using Green's theorem and assuming smoothness, the formal
Lagrangian $\lagr$ can be rewritten as:
	\begin{equation}\begin{aligned} \label{eq:lin_lagr_uweak}
 	\lagr (y,&\gamma,u,r,s) 
		= \costF + \int_\Omega \divg\of{A\sbr{\gamma}\grad r}\del{y+v} \dif x  \\
			&\quad + \int_{\bdy\Omega} \del{rA\sbr\gamma\grad\del{y+v} - \del{y+v}A\sbr\gamma\grad r}\cdot\nu \dif\sigma \\	
			&\quad + \int_0^1 \del{ \kappa \gamma \totalD^2_{x_1}s - s A\sbr\gamma\grad\del{y+v}\cdot\nu +us} \dif\sigma 
				+ \eval{\kappa \del{\totalD_{x_1}\gamma s - \gamma \totalD_{x_1}s}}_0^1.
	\end{aligned}\end{equation}
Imposing $\pair{\Dif_{y}\lagr{\optimal},h}  = 0$  to \eqref{eq:lin_lagr_uweak} implies that for every $h \in \sobZ 1 p \Omega$
	\begin{subequations} \label{eq:lin_rop_prev}
		\begin{equation}\begin{aligned} 
		     -\int_\Omega \divg\of{A\sbr \gop \grad \rop}h  \dif x
			&=  \int_\Omega \del{\yop +v - y_d} \del{1+\gop} h \dif x  \\
	    	     	&\quad + \int_{\bdy\Omega} \rop A\sbr\gop \grad h \cdot \nu \dif\sigma
	    	    		     - \int_0^1 \sop A\sbr\gop \grad h \cdot \nu \dif\sigma.
	         \end{aligned}\end{equation}	
	    Next, without loss of generality ($\CZinf\Omega$ is dense in $\sobZ{1}{p}{\Omega}$), we obtain
		\begin{align}
	               - \int_\Omega \divg\of{A\sbr \gop \grad \rop} h \dif x 
	               &=  \int_\Omega \del{\bar y+v- y_d} \del{1+\gop} h \dif x 
	                                && \forall h \in \CZinf\Omega,
		\end{align} 
		whereas, using that $A\sbr\gamma \grad h \cdot \nu$ can be chosen arbitrarily on $\bdy\Omega$ we deduce from 
		(\ref{eq:lin_rop_prev}a) and (\ref{eq:lin_rop_prev}b) that
                	\begin{align}
                       \rop - \sop\restriction_\Gamma = 0, \qquad \rop \restriction_\Sigma = 0. 
                	\end{align}    
	\end{subequations} 		
		In view of (\ref{eq:lin_rop_prev}b-c), the strong form of the boundary value problem for $\rop$ is: 		
		seek $\rop \in \sob{1}{q}{\Omega}$	such that
		 \begin{equation}\label{eq:lin_rop}
		 \left\{
			\begin{aligned}
				-\divg\of{A\sbr{\bar\gamma}\grad\rop} &=  \del{\yop + v - y_d}\del{1+\gop}  && \env{in } \Omega \\
					\rop &= \sop && \env{on } \Gamma \\ 
					\rop &= 0 && \env{on } \Sigma.
			\end{aligned}\right.
		\end{equation}
		
Next we employ the same technique to obtain the equations for the second adjoint variable $\sop$: we impose $\pair{\Dif_{\gamma}\lagr{\optimal},h}  = 0$  to \eqref{eq:lin_lagr_uweak} and make use of
the boundary conditions in \eqref{eq:lin_rop} to obtain for every $h \in \sobZ 1 \infty {\I}$
		\begin{align*}
			-\int_0^1\kappa \totalD^2_{x_1}s h\dif\sigma 
				&=  \int_0^1 \del{\gop - \gamma_d}h \dif\sigma \\
				&\quad+  \frac{1}{2} \int_\Omega \abs{\yop + v - y_d}^2 h \dif x 
				-\int_\Omega \Dif A\sbr{\gop}\difdir{h}\grad\del{\yop + v}\cdot \grad\rop \dif x.
		\end{align*} 
Therefore, the strong form of the boundary value problem for $\sop$ is: 
seek $\sop \in \sobZ{1}{1}{{\I}}$ 
		\begin{equation}\label{eq:lin_sop}
		\left\{ 
			\begin{aligned}
				-\kappa \totalD^2_{x_1}\sop 
				    &= \del{\gop - \gamma_d} 
				     +  \frac{1}{2}\int_0^1 \abs{\yop+v-y_d}^2 \dif x_2 
			            - \int_0^1 A_1\sbr{\gop}\grad\del{\yop + v} \cdot \grad \rop\dif x_2 \\
				    &\quad
+ \totalD_{x_1}\del{\int_0^1 A_2\sbr{\gop}\grad\del{\yop+v} \cdot \grad \rop\dif x_2}
						\quad \env{in } \I \\
				\sop(0) &= \sop(1) = 0,
			\end{aligned}\right.
		\end{equation}
where $A_1, A_2$ denote the representation of $DA$ given in \eqref{eq:op_DA}. We note that the integrals on the right 
hand side of \eqref{eq:lin_sop} correspond to integration in $x_2$ (vertical) direction. 
	
Finally, (\ref{eq:first_order_opt}b) implies	
\begin{align} \label{eq:lin_uop_var_ineq}
   \pairLt{\lambda \uop + \sop, u-\uop }{\I} \geq 0, \qquad \forall u \in \mathcal{U}_{ad}.
\end{align}   
To summarize, the solution $\optimal$ to the first-order optimality system \eqref{eq:first_order_opt}
satisfies \eqref{eq:lin_st_const}, \eqref{eq:lin_rop}, \eqref{eq:lin_sop} and \eqref{eq:lin_uop_var_ineq}. 
We stress that the formal approach presented in this section is very systematic and highly
useful even though it is not clear at the moment how to show the existence and (local) uniqueness of the optimal control
$\uop$.  A rigorous analysis will be developed in the next two sections.

%% file: 04_Control_to_State/g_map.tex
\section{The Control-to-state Map $G_v$}
\label{s:g_map}
Let $G_v$ denote the nonlinear map
\begin{align}   \label{eq:cont_to_state}
    \fullfunction{G_v}{\calU} {\bbW^1}
    {u}{\del{\gamma,y}},
\end{align}
where $\bbW^1 := \sobZ{1}{\infty}{\I}\times\sobZ{1}{p}{\Omega}$, $(\gamma,y)$ solves \eqref{eq:lin_st_const}, and the subscript on $G_v$ denotes dependence
on a fixed and non-trivial $v \in \sob{1}{p}{\Omega}$. Furthermore, 
$\calU \subset L^2(\I)$ is an open ball containing the set of admissible
controls $\Uad$, 
\[
    \Uad \subset \calU \subset L^2(\I), 
\]
which will be precisely specified in Definition~\ref{defn:const_set}.
Our goal is to show the existence of a control, derive the 
first-order necessary and second-order sufficient optimality conditions within the realm of a 
rigorous mathematical framework. The first-order optimality 
conditions requires to show  that $G_v$ is Fr\'echet differentiable (\autoref{s:g_differentiable}) 
and the second  order conditions require $G_v$ to be twice Fr\'echet differentiable 
(\autoref{s:g_twice_diff}).

The steps described above are standard for PDE-constrained optimization in fixed domains \cite{FTroltzsch_2010a}, but our analysis for the linearized curvature OC-FBP is novel. The novelty resides in the highly nonlinear structure of the underlying FBP, which is posed in a pair of Banach spaces one being non-reflexive, and yet we deal with \emph{minimal regularity}. 
A number of other control problems for FBPs fall under a similar functional framework \cite{MElliott_MHinze_VStyles_2007a, MElliott_YGiga_MHinze_VStyles_2010a}, but their theory is not as complete and conclusive as ours. This appears to be an area of intense current research.

The first step in this voyage is to show that there exists a  unique weak solution to \eqref{eq:lin_st_const}, which implies that $G_v$ is a well defined one-to-one nonlinear operator. 
In fact, it is known \cite{PSaavedra_RScott_1991} 
that for $u=0$ and $v$ small, a fixed point argument asserts the existence and uniqueness of a weak solution 
$\del{\gamma,y}$ in $\bbW^1$ to \eqref{eq:lin_st_const}. We further extend this analysis to the case where $u \neq 0$. This gives us an 
open ball $\calU \subset \Ltwo{\I}$ where we can show the existence of solution to  \eqref{eq:lin_st_const}.

%% file: 04_Control_to_State/g_well_posed.tex
\subsection{Well-posedness of the State System \eqref{eq:lin_st_const}}
\label{s:g_well_posed}
The weak form of the system \eqref{eq:lin_st_const} is: 
find $(\gamma,y) \in \bbW^1$ such that
	\begin{equation}\left\{\begin{aligned}\label{eq:g_var_pde} 
		\BO\sbr{y+v,z;A\sbr{\gamma}} &= 0  && \forall z \in	\sobZ{1}{q}{\Omega} \\ 
		\BG\sbr{\gamma,\zeta} + \BO\sbr{y+v,E\zeta;A\sbr{\gamma}} &= 
		\pairSob{u,\zeta}{\infty}{1}{1}{\I} 
			&&	\forall \zeta \in \sobZ{1}{1}{\I},
	\end{aligned}\right.\end{equation}
where $\BG:\sobZ 1 \infty {\I} \times \sobZ 1 1 {\I} \to \bbR$, $\BO:\sobZ1p\Omega\times\sobZ1q\Omega\to\bbR$ are defined by 
	\begin{equation}\begin{aligned} \label{eq:bilinear_form}
    		\BG\sbr{\gamma,\zeta} &:= \kappa \int_0^1 {\totalD_{x_1}\gamma}(x_1) {\totalD_{x_1}\zeta}(x_1) \dif x_1, \\
    		\BO\sbr{y,z;A\sbr{\gamma}} &:= \int_\Omega A\sbr{\gamma} \grad y \cdot \grad z \dif x.
	\end{aligned}\end{equation}
Furthermore, $E:\sobZ{1}{1}{\I} \rightarrow \sob{1}{q}{\Omega}$, $1<q<2$ 
denotes a continuous extension such that $E\zeta\restriction_{\Gamma}=\zeta$, $E\zeta\restriction_{\Sigma} = 0$. In particular, this implies the existence of
a constant $C_E \ge 1$, which dependes on $\Omega$ and $q$ and blows up as $q$ approaches 2 
\cite[Lemma 2]{PSaavedra_RScott_1991}, such that
	\begin{align}  \label{eq:w1_wq_ext}
    	\normSZ{E\zeta}{1}{q}{\Omega} &\leq C_E \normSZ{\zeta}{1}{1}{\I},  \quad
    		\forall \zeta \in \sobZ{1}{1}{\I}.
	\end{align}
Moreover, when $u \in \calU \subset \Ltwo{\I}$ and the test function 
$\zeta \in \sobZ{1}{1}{\I}$, then  $\zeta \in L^2(\I)$ and we may write 
	\begin{align} \label{eq:u_zeta_pair} 
        \pairSob{u,\zeta}{\infty}{1}{1}{\I} = \int_0^1 u \zeta ,
	\end{align}
where $\sob{-1}{\infty}{\I}$ is the dual space of $\sobZ{1}{1}{\I}$; we refer to \cite{RAAdams_JJFFournier_2003}. 
Since $\|\zeta\|_{L^\infty(\I)} \le |\zeta|_{\sobZ11{\I}}$,
this also enables us to deduce that for $u \in \Ltwo{\I}$
	\begin{align}\label{eq:u_negative_to_l2}
		\normS{u}{-1}\infty{\I} 
		\leq \normLt{u}{\I}.
	\end{align}
We will make use of these two facts repeatedly throughout the rest of the 
paper.

\begin{prop}[$\inf$-$\sup$ conditions]\label{prop:b_infsup} The following conditions hold for the bilinear forms $\BG\sbr{\cdot,\cdot}$ and $\BO\sbr{\cdot,\cdot;A\sbr{\gamma}}$ defined in \eqref{eq:bilinear_form} :
\begin{enumerate}[(i)]
\item\label{item:b_infsup_i} $\BG\sbr{\cdot,\cdot}:\sobZ 1 \infty {\I} \times \sobZ 1 1 {\I} \to \bbR$  is continuous and there exists a constant $\alpha > 0$ such that  for every $\gamma \in \sobZ1\infty{\I}$ and $s \in \sobZ11{\I}$
	\begin{subequations}
		\begin{align}\label{eq:b1_infsup_a}
			\normSZ{\gamma}{1}{\infty}{\I} &\leq \alpha \sup_{0\neq \zeta\in\sobZ{1}{1}{\I}}
				\frac{\BG\sbr{\gamma,\zeta}}{\normSZ{\zeta}{1}{1}{\I}} ,  \\
			\label{eq:b1_infsup_b}
			\normSZ{s}{1}{1}{\I} &\leq \alpha \sup_{0\neq \zeta\in\sobZ{1}{\infty}{\I}}
				\frac{\BG\sbr{\zeta,s}}{\normSZ{\zeta}{1}{\infty}{\I}} . 
		\end{align}
	\end{subequations}
\item\label{item:b_infsup_ii} If $\gamma \in \sobZ{1}{\infty}{\I}$ satisfies \eqref{eq:state_const}, then $\BO\sbr{\cdot,\cdot;A\sbr{\gamma}}:\sobZ1p\Omega\times\sobZ1q\Omega\to\bbR$ is continuous and there exist constants $P,Q$ with  $Q < 2 < P$ and $\beta > 0$, such that for $p \in \intoo{Q,P}$ and for all $y \in \sobZ1p\Omega$ 
	\begin{align}\label{eq:b2_infsup}
		\normSZ{y}{1}{p}{\Omega} &\leq \beta \sup_{0\neq z\in
		\sobZ{1}{q}{\Omega}}\frac{\BO\sbr{y,z;A\sbr{\gamma}}}{\normSZ{z}{1}{q}{\Omega}}.
        \end{align}
\end{enumerate}	
\end{prop}
\begin{proof}
For \eqref{eq:b1_infsup_a} and \eqref{eq:b2_infsup} we refer to \cite[Propositions
2.2-2.3]{PSaavedra_RScott_1991} and \cite{NGMeyers_1963} for a proof. For \eqref{eq:b1_infsup_b} we proceed as follows: 
applying the definition of the $L^1$-norm and the homogeneous Dirichlet values of $s$, we obtain
	\begin{align*}
		\normSZ{s}11{\I} &= \int_0^1 \abs{\totalD_{x_1}s} = \int_0^1 \sgn(\totalD_{x_1}s) \totalD_{x_1}s  
	\end{align*}
	Using the fact that $s\in \sobZ11{\I}$, we get $\int_0^1 \totalD_{x_1} s  = 0$, whence 
	\begin{align*}	
	    \normSZ{s}11{\I}
			&= \int_0^1 \underbrace{\del{\sgn(\totalD_{x_1}s) - \int_0^1 \sgn(\totalD_{x_1}s)}}_{=\totalD_{x_1}\zeta} \totalD_{x_1}s 
			= \frac{1}{\kappa} \BG\sbr{\zeta, s},  
	\end{align*}
where $\zeta(x_1) = \int_0^{x_1} \del{\sgn(\totalD_{x_1}s) - \int_0^1\sgn(\totalD_{x_1}s)} \in \sobZ1\infty{\I}$.  Estimate \eqref{eq:b1_infsup_b}  follows by noting that $\normSZ{\zeta}1\infty{\I} \leq 2$, and taking the $\sup$ over every $\zeta \in \sobZ 1 \infty {\I}$. 
\end{proof}

The following lemma demonstrates how one can improve the integrability index of a solution to a PDE obtained by standard methods.
\begin{lem}[improved integrability]\label{lem:baneso} Let $\Omega$ be an open Lipschitz bounded domain of $\bbR^d$ and  $\calB : \sobZ1\infty\Omega \times \sobZ11\Omega \to \bbR$ be a continuous bilinear form. Furthermore, suppose that 
	\begin{enumerate}[(i)]
		\item there exists $\alpha > 0$ such that 
			\begin{align} \label{eq:baneso_inf_sup_c}
				\normSZ{\chi}{1}{\infty}{\Omega} \leq \alpha \sup_{0\neq \psi\in\sobZ{1}{1}{\Omega}}
					\frac{\calB\sbr{\chi,\psi}}{\normSZ{\psi}{1}{1}{\Omega}} \qquad  \forall \chi \in \sobZ1\infty\Omega, 
			\end{align}
		\item and $\calB$ is continuous and coercive in $\sobZ{1}2\Omega$.
	\end{enumerate}
Then for every $F \in \sobZ11\Omega^*$, there exists a unique $\chi \in \sobZ{1}{\infty}\Omega$ such that 
	\begin{align}\label{eq:baneso_well_posed}
		\calB\sbr{\chi,\psi} = F(\psi) \mbox{ for all }  \psi \in \sobZ{1}{1}{\Omega} \mbox{ and } 
		\normSZ{\chi}1\infty\Omega \leq \alpha \normSD{F}11\Omega. 
	\end{align}
\end{lem}
\begin{proof}Since $\sobZ{1}{2}{\Omega} \subset \sobZ{1}{1}{\Omega}$, it follows that $F \in \sobZ{1}{1}{\Omega}^* \subset \sobZ{1}{2}{\Omega}^*$. The Lax-Milgram lemma guarantees the existence and uniqueness of $\chi \in \sobZ{1}{2}{\Omega}$ such that $\calB\sbr{\chi,\psi} = F(\psi)$ for all $\psi \in \sobZ{1}{2}{\Omega}$. 

Next, we extend $\calB\sbr{\chi,\cdot}$ as a linear functional on $\sobZ{1}{1}{\Omega}$. To this end, let $\seq{\psi_n}\subset \sobZ{1}{2}{\Omega}$ be a Cauchy sequence in the $\sobZ{1}{1}{\Omega}$-norm. It immediately follows that $\seq{\calB\sbr{\chi,\psi_n}}$ is also Cauchy in $\bbR$, i.e. $$\abs{\calB\sbr{\chi,\psi_n-\psi_m}} = \abs{F(\psi_n - \psi_m)} \leq \norm{F}_{\sobZ11\Omega^*}\normSZ{\psi_n - \psi_m}11\Omega.$$ 

Finally, by the density of $\sobZ{1}{2}{\Omega}$ in $\sobZ11\Omega$, not only do we obtain $\psi_n\to \psi \in \sobZ11\Omega$, but also  $$\calB\sbr{\chi, \psi} := \lim_{n\to\infty} \calB\sbr{\chi, \psi_n} = \lim_{n\to\infty} F(\psi_n) = F(\psi).$$ The estimate for $\normSZ{\chi}1\infty{\Omega}$ follows from \eqref{eq:baneso_inf_sup_c}.
\end{proof}

\subsubsection{First-order Regularity}
Now we are ready to prove that there exists a unique solution to \eqref{eq:g_var_pde} with first-order regularity. 
Since the system \eqref{eq:g_var_pde} is nonlinear we will obtain this result by applying the Banach fixed point theorem 
combined with a smallness assumption on a non-trivial $v$. To this end, we let $2 < p < P$ and equip the space 
$\bbW^1 = \sobZ 1 \infty {\I} \times\sobZ1 p \Omega$ with the equivalent norm
\begin{align}  \label{eq:W1_equiv_norm}
 \norm{\del{\gamma,y}}_{\bbW^1} :=\del{1+\beta C_A}\normSZ{v}1p\Omega \normSZ{\gamma}{1}{\infty}{{\I}} +
		\normSZ{y}{1}{p}{\Omega},
\end{align}		
where $C_A$ and $\beta$ are given in  \eqref{eq:A_unif_bound} and \eqref{eq:b2_infsup}, 
and define the closed (convex) ball 
\begin{align}    \label{eq:W1_convex_set}
    \bbB_v  := \set{(\gamma,y) \in \bbW^1 :\;
		\normSZ{y}{1}{p}{\Omega} \leq \beta C_A \normSZ v 1 p \Omega,\; \normSZ{\gamma}1\infty{\I} \leq 1}.
\end{align}		
Furthermore, consider the operator $T:\bbB_v \rightarrow \bbW^1$ defined as
	\begin{align} \label{eq:T_oper}
		T(\gamma, y) &:=  \del{T_1(\gamma, y), T_2(\gamma, y)} = \del{\gtl, \ytl} && \forall \del{\gamma,y} \in \bbB_v,
	\end{align}
where $\gtl = T_1(\gamma, y) \in \sobZ1\infty{\I}$ satisfies for every  $\zeta \in \sobZ{1}{1}{\I}$
	\begin{equation}\label{eq:t1_var}
		\BG\sbr{\gtl,\zeta} = -\BO\sbr{y + v,E\zeta;A\sbr{\gamma}} + \pairSob{u,\zeta}{\infty}{1}{1}{\I},
	\end{equation}
and $\ytl = T_2(\gamma,y) \in \sobZ1p\Omega$ satisfies for every $z \in \sobZ{1}{q}{\Omega}$
	\begin{equation}\label{eq:t2_var}
		\BO\sbr{\ytl + v,z;A\sbr{T_1\del{\gamma,y}}} = 0 .
	\end{equation}
With these definitions at hand we proceed to find conditions under which $T$ not only maps $\bbB_v$ into itself but is in fact a contraction in $\bbB_v$.

\begin{lem}[range of $T$] \label{lem:t_maps_weps}
Let $T_1$ and $T_2$ be the operators defined in \eqref{eq:t1_var} and \eqref{eq:t2_var}, and $C_A$ and $C_E$ be the constants defined in \eqref{eq:A_unif_bound} and \eqref{eq:w1_wq_ext}. Furthermore, suppose there exists $\theta_1 \in \del{\beta C_A/(1+\beta C_A),1}$ such that 
	\begin{align}\label{eq:v_invariant}
		\normSZ{v}1p\Omega \leq \del{1-\theta_1}\del{\alpha C_E C_A\del{1+\beta C_A}}^{-1}. 
	\end{align}
If $u \in \Ltwo{\I}$ with $\normLt{u}{\I} \leq \theta_1/\alpha$, then the range of $T$ is contained in $\bbB_v$.

\end{lem}
\begin{proof} Let $\del{\gamma, y} \in \bbB_v$ be fixed but arbitrary. First we rely on \lemref{lem:baneso} to show the well-posedness of $T_1$. Since it is straight-forward to check that 
$\BG$ is continuous and coercive in $\sobZ12{\I}$, we only need to show the regularity of the forcing term in \eqref{eq:t1_var}. If we define 
$F(\zeta) := -\BO\sbr{y + v,E\zeta;A\sbr{\gamma}} + \pair{u,\zeta}$  and use \eqref{eq:A_unif_bound}, \eqref{eq:w1_wq_ext} and \eqref{eq:W1_convex_set} we find that
\begin{align} \label{eq:chi_w1}
   \abs{F(\zeta)} 
   			&\le C_A \del{\normSZ{y}{1}{p}{\Omega} + \normSZ{v}{1}{p}{\Omega}} \normSZ{E\zeta}{1}{q}{\Omega} 
   				+\normLt{u}{\I}\normSZ{\zeta}{1}{1}{\I} \nonumber\\ 
   			&\le \del{C_E C_A \del{1+\beta C_A}\normSZ{v}{1}{p}{\Omega}
   				+ \normLt{u}{\I}}\normSZ{\zeta}{1}{1}{\I}, 
\end{align}
whence $F \in \sobZ11{\I}^*$ and we conclude from \eqref{eq:baneso_well_posed} that
\begin{align*}
       \normSZ{\gtl}{1}{\infty}{\I} 		
    &\leq \alpha\del{C_E C_A \del{1+\beta C_A}\normSZ{v}{1}{p}{\Omega}
   				+ \normLt{u}{\I}} 
    \leq \del{1-\theta_1} + \theta_1  = 1.
\end{align*}		   

The well-posedness of $T_2$ follows by \propref{prop:b_infsup} and
 the Banach-Ne\v{c}as theorem for reflexive Banach spaces \cite[Theorem 2.6]{AErn_JLGuermond_2004a}. Applying \eqref{eq:b2_infsup} we obtain
	\[
    	\normSZ{\ytl}{1}{p}{\Omega} \leq  \beta C_A \normSZ{v}{1}{p}{\Omega}.
	\] 
Since $\del{\gamma, y}$ is arbitrary, we conclude that the range of $T$ is contained in $\bbB_v$.
\end{proof}


\begin{defn}[control sets $\calU$ and $\Uad$]\label{defn:const_set}
Let $\theta_1$ be as in \lemref{lem:t_maps_weps}. We define 
the (nontrivial) open ball $\calU \subset L^2(\I)$ as
\begin{equation}    \label{eq:open_ball}
\calU := \set{ u \in \Ltwo{\I}\;:\; \normLt{u}{\I} < \theta_1 /\alpha} ,
\end{equation}
and the admissible set of controls $\Uad$ as the (nontrivial) closed ball
\begin{equation}    \label{eq:v_u_const_set}
\Uad := \set{ u \in \calU \;:\; \normLt{u}{\I} \leq \theta_1 /2\alpha}.
\end{equation}
\end{defn}

We may wonder about the presence of $\calU$ in Definition~\ref{defn:const_set}. This will enable us to prove the
Fr\'echet differentiability of $G_v$ at any $u \in \Uad$ later in \secref{s:g_first_deriv}. In the next theorem we will show that the state equations are solvable for any $ u \in \calU$.

\begin{thm}[$T$ is a contraction]\label{thm:t_contraction} 
Let the assumptions of Lemma~\ref{lem:t_maps_weps} hold and suppose further that there exists a $\theta_2 \in \intoo{0,1}$ such that 
	\begin{align} \label{eq:v_contraction}
		\normSZ{v}1p\Omega \leq (1-\theta_2) \del{\alpha C_E C_A\del{1+\beta C_A}^2}^{-1}.
	\end{align}
Then, the map  $T$ defined in \eqref{eq:T_oper} is a contraction in $\bbB_v$ with constant $1-\theta_2$ for all $u \in \calU$.
\end{thm}
\begin{proof} 
	Consider $(\gamma_1,y_1), (\gamma_2,y_2) \in \bbB_v$ such that  $\del{\gamma_1,y_1} \neq \del{\gamma_2,y_2}$.
	Using \eqref{eq:T_oper} we have that $T\del{\gamma_i,y_i} = \del{\gtl_i,\ytl_i}$ solves 
	\eqref{eq:t1_var} and \eqref{eq:t2_var} for $i=1,2$. 
	Therefore, combining \propref{prop:b_infsup} (\ref{item:b_infsup_i}) and \lemref{lem:t_maps_weps} with \eqref{eq:v_contraction} implies
	\begin{eqnarray}\label{eq:t_c_gamma}
			\normSZ{\gtl_1 - \gtl_2}{1}{\infty}{\I} 
				&&\leq \alpha \sup_{\normSZ{\zeta}{1}{1}{\I}=1} \BG\sbr{\gtl_1 - \gtl_2, \zeta}
					\nonumber \\ 
				&&= \alpha \sup_{\normSZ{\zeta}{1}{1}{\I}=1} 
					\BO\sbr{y_2-y_1,E\zeta;A\sbr{\gamma_1}} + \BO\sbr{y_2+v,E\zeta;A\sbr{\gamma_2}-A\sbr{\gamma_1}}
					\nonumber \\ 
				&&\leq \alpha C_EC_A\del{\normSZ{y_1-y_2}{1}{p}{\Omega} 
					+ \del{1+\beta C_A}\normSZ{v}1p\Omega\normSZ{\gamma_1-\gamma_2}{1}{\infty}{\I}}
					\nonumber \\ 
				&&=\alpha C_EC_A\norm{\del{\gamma_1-\gamma_2,y_1-y_2}}_{\bbW^1}.
	\end{eqnarray}	
         Similarly, \propref{prop:b_infsup} (\ref{item:b_infsup_ii}) in conjunction with \eqref{eq:t2_var} leads to
         \begin{eqnarray}\label{eq:t_c_y}
			\normSZ{\ytl_1 - \ytl_2}{1}{p}{\Omega}
				&&\leq \beta\sup_{\normSZ{z}{1}{q}{\Omega}=1} 
					\BO\sbr{\ytl_1 - \ytl_2,z;A\sbr{\gtl_1}} \nonumber \\
				&&= \beta\sup_{\normSZ{z}{1}{q}{\Omega}=1} 
					\BO\sbr{\ytl_2+v,z;A\sbr{\gtl_2}-A\sbr{\gtl_1}} \\
				&&\leq \beta C_A\normSZ{\ytl_2+v}1p\Omega \normSZ{\gtl_1 - \gtl_2}{1}{\infty}{\I}
					\nonumber \\ 
				&&\leq \beta C_A\del{1+\beta C_A }\normSZ{v}1p\Omega\normSZ{\gtl_1 - \gtl_2}{1}{\infty}{\I}. 
					\nonumber
          \end{eqnarray}
          Finally, \eqref{eq:t_c_gamma} and \eqref{eq:t_c_y} yield
          \begin{align*}
          	\norm{\del{\gtl_1, \ytl_1} - \del{\gtl_2,\ytl_2}}_{\bbW^1}
             		&\leq \del{1+\beta C_A}^2\normSZ{v}1p\Omega\normSZ{\gtl_1 - \gtl_2}1\infty{\I} \\
			&\leq \alpha C_E C_A\del{1+\beta C_A}^2\normSZ{v}1p\Omega\norm{\del{\gamma_1, y_1}-\del{\gamma_2,y_2}}_{\bbW^1} \\
			&\leq \del{1-\theta_2} \norm{\del{\gamma_1, y_1}-\del{\gamma_2,y_2}}_{\bbW^1},
          \end{align*}
          where the last inequality follows from \eqref{eq:v_contraction}. Since $\theta_2 \in \intoo{0,1}$, $T$ is a contraction with constant $1-\theta_2$ as asserted.
\end{proof}

We point out that the state constraint \eqref{eq:state_const} is used in the 
proof of \thmref{thm:t_contraction} at two distinct instances. The first is to 
estimate $A\sbr{\gamma_2}-A\sbr{\gamma_1}$ and $A\sbr{\gtl_2}-A\sbr{\gtl_1}$. 
The second use
is to invoke the inf-sup constant $\beta$ for $y$; see \eqref{eq:b2_infsup}.
For details on how $\beta$ depends on the state constraint we refer to \cite[Proposition 2.3]{PSaavedra_RScott_1991}.

\begin{cor}[well-posedness of state system] \label{cor:g_operator} 
For every $u\in \calU$, the open ball of \defnref{defn:const_set}, and $v$ satisfying \eqref{eq:v_invariant} and \eqref{eq:v_contraction}, there exists a unique solution 
$G_v(u) = \del{\gamma,y} \in \bbW^1$ to the state equations \eqref{eq:g_var_pde}. This further implies that 
$G_v:\calU\to\bbW^1$ is a well defined, one-to-one, nonlinear operator. 
\end{cor}

\begin{proof} Let $u \in \calU$ be fixed but arbitrary. It now follows that $T$ is a contraction in the closed convex set $\bbB_v$ (cf. Theorem~\ref{thm:t_contraction}) and applying the Banach fixed point theorem we obtain a unique $\del{\gamma,y} \in \bbB_v$ such that $T\del{\gamma, y} = \del{\gamma,y}$. In view of \eqref{eq:t1_var} and \eqref{eq:t2_var}, this is equivalent to saying that $\del{\gamma,y}$ is the weak solution to the FBP \eqref{eq:g_var_pde}, i.e. $G_v(u) = \del{\gamma,y}$.
\end{proof}

%% file: 04_Control_to_State/g_interpolation.tex
\subsubsection{Enhanced Regularity of $\gamma$}
\label{s:g_intepolation}
Corollary~\ref{cor:g_operator} implies  the existence and uniqueness of a solution $(\gamma,y)$ to \eqref{eq:lin_st_const} with first-order regularity, provided $u\in \calU$ and $v$ satisfies 
\eqref{eq:v_invariant} and \eqref{eq:v_contraction}. That is,  we only have one weak derivative for $\gamma$ and $y$.  
In the sequel we will show that the solution $\del{\gamma,y} = G_v(u)$ is slightly more regular \emph{without} any extra restrictions on $u$ or $v$. More specifically, we will show that 
	\begin{equation}\label{eq:gamma_enhanced_regularity}
		\gamma \in \sob{1+1/q}{p}{\I} \cap \sobZ1\infty{\I}. 
	\end{equation}
The importance of this result will be evident in \autoref{s:oc_rc_existence} where the existence of an optimal control is proven.  Despite its importance, the proof is rather simple. 

Let $\del{\gamma, y} \in \bbW^1$ be a weak solution to \eqref{eq:g_var_pde}. The function $\gamma$ satisfies
\[
	-\kappa \totalD^2_{x_1}\gamma = {-A\sbr{\gamma}\grad\del{y+v}\cdot\nu + u} =: f ,
\]
in the sense of distributions. 
If we assume, for the moment, that $f \in \sob{-1/p}{p}{\I}$, then
$\totalD^2_{x_1}\gamma \in \sob{-1/p}{p}{{\I}}$. This directly implies
$\gamma \in \sob{2-1/p}{p}{\I}$, i.e. $\gamma \in \sob{1+1/q}p{\I}$ as
asserted. It thus remains to show 
$A\sbr{\gamma}\grad\del{y+v}\cdot\nu\in\sob{-1/p}{p}{\I}$ because
$u \in \Ltwo{\I}$. 

Given $\phi \in \sobZ{1/p}{q}{\I}$, which we identify with 
$\sobZ{1/p}{q}{\Gamma}$, we extend $\phi$
by zero to $\Sigma$. We note that $\phi\in\sob{1/p}{q}{\partial\Omega}$
and that in fact $\sobZ{1/p}{q}{\I} = \sob{1/p}{q}{\I}$ for $q < 2$,
according to Lions-Magenes \cite[Th\'eor\`em 3.1]{JLLions_EMagenes_1961c}.
We can further view $\phi$ as the trace of the extension $E\phi$ to
$\Omega$ so
that $\normS{E\phi}1q\Omega = \normS{\phi}{1/p}{q}{\Gamma}$. With this in mind, 
\begin{align*}
	\pairSob{A\sbr{\gamma}\grad\del{y+v}\cdot\nu, \phi}{p}{1/p}{q}{\Gamma} 
		&= \int_\Omega A\sbr{\gamma}\grad\del{y+v}\cdot\grad{E\phi},
\end{align*}
whence
	\begin{align*}
		\normS{A\sbr{\gamma}\grad\del{y+v}\cdot\nu}{-1/p}p\Gamma \leq C_EC_A\del{1+\beta C_A}\normSZ{v}1p\Omega.
	\end{align*}

We collect this result in the next theorem.

\begin{thm}[enhanced regularity]\label{thm:g_interpolation} 
Let $u\in \calU$ and $v$ satisfy \eqref{eq:v_invariant} 
and \eqref{eq:v_contraction}. 
If $G_v(u) = \del{\gamma, y} \in \bbW^1 $ is given in
Corollary~\ref{cor:g_operator}, then $\gamma \in \sob{1+1/q}{p}{\I}
\cap \sobZ{1}{\infty}{\I}$.
\end{thm}

%% file: 04_Control_to_State/g_lipschitz.tex
\subsection{$G_v$ is Lipschitz Continuous}
\label{s:g_lipschitz}
The first step to show that $G_v$ is twice Fr\'echet differentiable is to demonstrate that it is Lipschitz continuous.

In the interest of saving some space we will rewrite the variational system \eqref{eq:g_var_pde} in the following  form: find $\del{\gamma, y} \in \bbW^1$ such that for every $\del{\zeta, z} \in \sobZ11{\I} \times \sobZ1q\Omega$
	\begin{align} \label{eq:g_var_pde_one}
		\BG\sbr{\gamma, \zeta} + \BO\sbr{y+v, z + E\zeta; A\sbr{\gamma}} &= \pairSob{u,\zeta}{\infty}{1}{1}{\I}.
	\end{align} 
With this new notation in place we are ready to study the Lipschitz continuity of $G_v$.
\begin{thm}[Lipschitz continuity of $G_v$]  \label{thm:g_lipschitz} 
If $v$ fulfills  the conditions of Corollary~\ref{cor:g_operator}, then $G_v$ satisfies 
	\begin{align} \label{eq:g_lipschitz}
   		\norm{G_v(u_1) - G_v(u_2)}_{\bbW^1} \leq L_G\normLt{u_1-u_2}{\I} &&\forall u_1, u_2 \in \calU,
	\end{align}
with constant $L_G = \frac{\alpha}{\theta_2}\del{1+\beta C_A}^2\normSZ{v}1p\Omega$.   
\end{thm}
\begin{proof} 
Given $u_1, u_2 \in \calU$, set $ \del{\gamma_1,y_1} - \del{\gamma_2, y_2} = G_v(u_1) - G_v(u_2)$. Using 
\eqref{eq:g_var_pde_one}, we have for every $\del{\zeta,z} \in \sobZ 1 1 {\I} \times \sobZ{1}{q}{\Omega}$
	\begin{align*}
	    	\BG\sbr{\gamma_1-\gamma_2,\zeta} &+ \BO\sbr{y_1+v,z + E\zeta;A\sbr{\gamma_1}} \\
	    			&- \BO\sbr{y_2+v,z+E\zeta;A\sbr{\gamma_2}} = \pairLt{u_1 - u_2,\zeta}{\I}.
	\end{align*}
	Subtracting $\BO\sbr{y_2+v,z+E\zeta;A\sbr{\gamma_1}}$ from both sides and rearranging terms yields
	\begin{align*}
	    \BG\sbr{\gamma_1-\gamma_2,\zeta} &+ \BO\sbr{y_1-y_2,z+E\zeta;A\sbr{\gamma_1}}  \\ 
	    	&= \BO\sbr{y_2+v,z+E\zeta;A\sbr{\gamma_2}-A\sbr{\gamma_1}} +  \pairLt{u_1 - u_2,\zeta}{{\I}}. 
	\end{align*}
	The $\inf$-$\sup$ estimates from \propref{prop:b_infsup}, together with $\del{\gamma_i, y_i} \in \bbB_v$ for $i=1,2$, imply for $\zeta = 0$
	\begin{align}\label{eq:g_lipschitz_y}
		\normSZ{y_1-y_2}{1}{p}{\Omega} 
			&\leq \beta C_A\normSZ{y_2+v}{1}{p}{\Omega} \normSZ{\gamma_1-\gamma_2}{1}{\infty}{{\I}} \nonumber \\ 
			&\leq \beta C_A\del{1+\beta C_A }\normSZ{v}1p\Omega\normSZ{\gamma_1-\gamma_2}{1}{\infty}{\I} ,
	\end{align}
	and for $z=0$
	\begin{align*}
		\normSZ{\gamma_1 - \gamma_2}{1}{\infty}{\I}
			&\leq \alpha C_EC_A\del{\normSZ{y_1-y_2}{1}{p}{\Omega}
				+ \normSZ{y_2+v}{1}{p}{\Omega}\normSZ{\gamma_1 - \gamma_2}{1}{\infty}{\I}}  \\
			&\quad + \alpha \normLt{u_1 - u_2}{\I}  \\ 
			&\leq \alpha C_EC_A\del{1+\beta C_A}^2\normSZ{v}1p\Omega \normSZ{\gamma_1 - \gamma_2}1\infty{\I} + 
			 	\alpha\normLt{u_1-u_2}{\I}.
	\end{align*}
	Finally, in view of \eqref{eq:v_contraction}, we infer that
	\begin{align} \label{eq:g_lipschitz_gamma}	
			\normSZ{\gamma_1 - \gamma_2}1\infty{\I} 
				\leq  \frac{\alpha}{\theta_2}\normLt{u_1 - u_2}{\I}.
	\end{align}
	The asserted estimate follows immediately from the definition of $\norm{\cdot}_{\bbW^1}$ in  \eqref{eq:W1_equiv_norm}.
\end{proof}

%% file: 04_Control_to_State/g_differentiable.tex
\subsection{$G_v$ is Fr{\'e}chet Differentiable}
\label{s:g_differentiable}
The next step towards showing the twice Fr{\'e}chet differentiability of $G_v$ entails analyzing
the well-posedness of the \emph{linear} variational system: find $\del{\gamma,y} \in \bbW^1$ 
such that for every $\del{\zeta,z} \in \sobZ11{\I} \times \sobZ1q\Omega$ 
	\begin{equation}
		\BG\sbr{\gamma, \zeta} + \DO\sbr{\del{\gamma, y},z + E\zeta; \gop, \yop}
			= F_\Omega\del{z+E\zeta} +  F_\Gamma\del{\zeta}
	\label{eq:g_diff_general}\end{equation}
where 
	$$
	\DO\sbr{\del{\gamma, y}, \cdot; \gop, \yop} := \BO\sbr{y,\cdot;A\sbr{\gop}} + \BO\sbr{\yop + v,\cdot;\Dif A\sbr\gop\difdir{\gamma}},$$ 
$\del{\gop,\yop} = G_v(\uop) \in \bbB_v$ for a fixed $\uop$ in $\calU$, $\Dif A\sbr\gop\difdir\gamma$ is given in \eqref{eq:op_DA}, and
$F_\Omega \in \sob{1}{q}{\Omega}^*$ and $F_\Gamma \in \sobZ{1}1{\I}^*$ are fixed but arbitrary.

%% file: 04_Control_to_State/g_diff_preliminaries.tex
\subsubsection{Preliminary Estimates}
\label{sec:G_prelim}
Given that the coupled system \eqref{eq:g_diff_general} is linear, one
would be inclined to use the standard Banach-Ne\v{c}as theorem to prove its
well-posedness directly. We deviate from this approach and resort to the machinery
already put in place.

Consider the operator $T:\bbW^1 \to \bbW^1$ given by
	\begin{equation}\label{eq:T_diff_oper}
		T\del{\gamma, y} := \del{T_1(\gamma,y), T_2(\gamma, y)} = \del{\ghat, \yhat}
			\quad \forall \del{\gamma, y} \in \bbW^1,
	\end{equation}
where $\ghat = T_1\del{\gamma, y} \in \sobZ1\infty{\I}$ satisfies for every $\zeta \in \sobZ{1}{1}{{\I}}$
	\begin{equation}\label{eq:t1_diff_var}
		\BG\sbr{\ghat,\zeta} = 
			- \DO\sbr{\del{\gamma, y}, E\zeta; \gop,\yop}
			 + F_\Omega\del{E\zeta} + F_\Gamma\del\zeta,
	\end{equation}
and $\yhat = T_2\del{\gamma,y} \in \sobZ1p\Omega$ satisfies for every $z \in \sobZ{1}{q}{\Omega}$
	\begin{align}\label{eq:t2_diff_var}
		\BO\sbr{\yhat, z; A\sbr{\gop}} = 
			- \BO\sbr{\yop + v, z; \Dif A\sbr\gop\difdir{T_1(\gamma, y)}} + 
			F_\Omega\del z.
	\end{align}
We point out that any fixed point of $T$ is also a solution to \eqref{eq:g_diff_general}. To infer the existence of a fixed point we exploit the linear structure of \eqref{eq:g_diff_general}. Therefore, it suffices to show the well-posedness of the intermediate operators $T_1$ and $T_2$, and to show that $T$ is a contraction in $\bbW^1$. 
\begin{lem}[well-posedness of $T_1$ and $T_2$] \label{lem:t_diff_maps_back}
Let $T_1$, $T_2$ be the operators defined in \eqref{eq:t1_diff_var} and \eqref{eq:t2_diff_var} with
$\del{\gop,\yop} \in \bbB_v$. The following holds
\begin{enumerate}[(i)]
	\item\label{item:t_diff_maps_back_b} for every $\del{\gamma,y} \in \bbW^1$, there exists a unique $\ghat = T_1\del{\gamma,y}$ satisfying \eqref{eq:t1_diff_var} and   
         \[
              \normSZ{\ghat}{1}{\infty}{{\I}} 
              	\leq \alpha\del{C_EC_A\norm{\del{\gamma,y}}_{\bbW^1}
			+ C_E\normSD{F_\Omega}{1}q\Omega
              		+ \normSZD{F_\Gamma}{1}{1}{{\I}}},
         \]
	\item\label{item:t_diff_maps_back_a} for every $\del{\gamma, y} \in \bbW^1$, there exists a unique $\yhat = T_2(\gamma, y)$ satisfying 
               \eqref{eq:t2_diff_var} and  
         \[ 
            \normSZ{\yhat}{1}{p}{\Omega} \leq
            	\beta C_A \del{1+\beta C_A}\normSZ{v}1p\Omega\normSZ{\ghat}{1}{\infty}{{\I}} 
	 				+\beta \normSD{F_\Omega}1q\Omega.
         \]   
\end{enumerate}
\end{lem}
\begin{proof} 
To prove (\ref{item:t_diff_maps_back_b}) we proceed as in \lemref{lem:t_maps_weps}.
It suffices to check that the right-hand-side $\env{RHS}(\zeta)$ of \eqref{eq:t1_diff_var} is in $\sobZ11{\I}^*$, namely
	\begin{align*}
		\abs{\env{RHS}(\zeta)} \leq 
			\bigg(C_EC_A\Big(\normSZ{y}{1}{p}{\Omega} 
				&+ \del{1+\beta C_A}\normSZ{v}1p\Omega\normSZ{\gamma}1\infty{\I}\Big) 
			    \\
				&+ C_E\normSD{F_\Omega}1q\Omega 
				+ \normSZD{F_\Gamma}11{\I} 
			 \bigg)\normSZ{\zeta}{1}{1}{{\I}} .
	\end{align*}
The desired estimate (i) follows from \lemref{lem:baneso} with the coercivity of $\BG$ in $\sobZ{1}{2}{{\I}}$,  
and the definition of $\norm\cdot_{\bbW^1}$ in \eqref{eq:W1_equiv_norm}.
The remaining estimate (\ref{item:t_diff_maps_back_a}) is a straightforward
application of Proposition \ref{prop:b_infsup}(ii).
\end{proof}

\begin{thm}[T is a contraction] \label{thm:t_contraction2}  Let \eqref{eq:v_contraction} hold for some $\theta_2 \in \intoo{0,1}$. The operator $T$ defined in \eqref{eq:T_diff_oper} is a contraction in $\bbW^1$ with constant $1-\theta_2$.
\end{thm}
\begin{proof} We proceed in a similar fashion to \thmref{thm:t_contraction}. 
Consider not identical $\del{\gamma_1, y_1}$ and $\del{\gamma_2, y_2}$ in $\bbW^1$, and use \eqref{eq:T_diff_oper} to write $\del{\ghat_i, \yhat_i} = T\del{\gamma_i, y_i}$ for $i=1,2$. Applying \lemref{lem:t_diff_maps_back} (\ref{item:t_diff_maps_back_b}), we obtain 
	\begin{align*}
		\normSZ{\ghat_1 - \ghat_2}1\infty{\I} \leq
			\alpha C_E C_A\norm{\del{\gamma_1 - \gamma_2, y_1-y_2}}_{\bbW^1}.
	\end{align*}
Similarly,  \lemref{lem:t_diff_maps_back} (\ref{item:t_diff_maps_back_a}) implies
	\begin{align*}
		\normSZ{\yhat_1 - \yhat_2}{1}{p}{\Omega} 
			&\leq \beta C_A\del{1+\beta C_A}\normSZ{v}1p\Omega\normSZ{\ghat_1 - \ghat_2}{1}{\infty}{{\I}}.
	\end{align*}
Lastly, the upper bound \eqref{eq:v_contraction} on $v$ yields
	\begin{align*}
		\norm{\del{\ghat_1 - \ghat_2, \yhat_1 - \yhat_2}}_{\bbW^1}
			&\leq \del{1+\beta C_A}^2\normSZ{v}1p\Omega\normSZ{\ghat_1-\ghat_2}1\infty{\I} \\			
			&\leq \del{1-\theta_2}\norm{\del{\gamma_1-\gamma_2, y_1 - y_2}}_{\bbW^1}.
	\end{align*}
Hence, $T$ is a contraction with constant $1-\theta_2$, as asserted.
\end{proof}
\begin{cor}[well-posedness of the linear system \eqref{eq:g_diff_general}] \label{cor:g_diff_prelim} Under the assumptions of \thmref{thm:t_contraction2}, there exists a unique solution $\del{\gamma,y} \in \bbW^1$ to the variational equation \eqref{eq:g_diff_general} and the following estimates hold
	\begin{align} 
		\label{eq:g_diff_prelim_continuity_gamma}
		\normSZ{\gamma}1\infty{\I} 
			&\leq \frac{\alpha}{\theta_2}\del{C_E\del{1+\beta C_A}\normSD{F_\Omega}1q\Omega + \normSZD{F_\Gamma}11{\I}} \\
		\label{eq:g_diff_prelim_continuity_y}
		\normSZ{y}1p\Omega
			&\leq
                        \frac{\beta}{\theta_2}\del{\normSD{F_\Omega}1q\Omega
                          + \alpha C_A\del{1+\beta
                            C_A}\normSZ{v}1p\Omega
                          \normSZD{F_\Gamma}11{\I}}.
	\end{align}
Therefore
\begin{equation*}
\begin{aligned}
		\norm{\del{\gamma,y}}_{\bbW^1}
			&\leq 
			\frac{1}{\theta_2}
				\del{\alpha C_E\del{1+\beta
                                    C_A}^2\normSZ{v}1p\Omega +
                                  \beta}\normSD{F_\Omega}1q\Omega  
                        \\
			&\quad
				+ \frac{\alpha}{\theta_2}\del{1+\beta
                                  C_A}^2\normSZ{v}1p\Omega\normSZD{F_\Gamma}11{\I}.
\end{aligned}
\end{equation*}
\end{cor}
\begin{proof}  Existence and uniqueness follows from
  \thmref{thm:t_contraction2}.  As far as the estimates go,  we will only derive \eqref{eq:g_diff_prelim_continuity_gamma} since the other two are mere consequences. 

To this end we apply \lemref{lem:t_diff_maps_back} and the upper bound \eqref{eq:v_contraction} for $v$ to get
	\begin{align*}
		\normSZ{\gamma}1\infty{\I} 
			&\leq \alpha C_E C_A\del{\del{1+\beta C_A}\normSZ{v}1p\Omega\normSZ{\gamma}1\infty{\I} + \normSZ{y}1p\Omega} \\
				&\quad +\alpha \del{C_E\normSD{F_\Omega}1q\Omega + \normSZD{F_\Gamma}11{\I}} \\
			&\leq \alpha C_E C_A\del{1+\beta C_A}^2\normSZ{v}1p\Omega\normSZ{\gamma}1\infty{\I} \\
				&\quad + \alpha C_E \del{1+\beta C_A}\normSD{F_\Omega}1q\Omega + \alpha\normSZD{F_\Gamma}11{\I}. \\
			&\leq \del{1-\theta_2}\normSZ{\gamma}1\infty{\I} 
				+ \alpha C_E \del{1+\beta C_A}\normSD{F_\Omega}1q\Omega + \alpha\normSZD{F_\Gamma}11{\I}.
	\end{align*}
The estimate \eqref{eq:g_diff_prelim_continuity_gamma} follows immediately. 
\end{proof}

%% file: 04_Control_to_State/g_first_order.tex
\subsubsection{The First-order Fr{\'e}chet Derivative}\label{s:g_first_deriv}
In this section we will prove the first-order differentiability of the control-to-state map $G_v$. We will frequently use the notation $\norm{f(h)} = \littleo{\normLt{h}{\I}}$ which is equivalent to $\lim_{\normLt{h}\I \rightarrow 0}\frac{\norm{f(h)}}{\normLt{h}\I} = 0$.

\begin{thm}[the Fr\'echet derivative of $G_v$]\label{thm:g_diff_u} 
The control-to-state map $G_v:\calU \to \bbW^1$ 
admits a first-order Fr\'echet derivative  
$G_v' : \calU \rightarrow \mathcal{L}\del{L^2(\I), \bbW^1}$. 
Hence for all $\uop \in \calU$ and 
all $h \in L^2(\I)$, $\del{\gamma, y} :=  G_v'\del{\uop}h \in \bbW^1$ satisfies the linear variational system \eqref{eq:g_diff_general} with $F_\Omega\del{\cdot;h} = 0$ and $F_\Gamma\del{\zeta;h} = \int_0^1 h \zeta$, namely
	\begin{equation} \label{eq:g_diff_u}
		\BG\sbr{\gamma, \zeta} + \DO\sbr{\del{\gamma, y}, z+ E\zeta; \gop, \yop} = \int_0^1 h \zeta, 
			\quad\quad \forall \del{\zeta, z} \in \sobZ11{\I} \times \sobZ1q\Omega.
	\end{equation}
Moreover, the following estimate holds
	\begin{align}\label{eq:g_diff_u_estimate}
		\norm{\del{\gamma,y}}_{\bbW^1} 
			&\le \frac{\alpha}{\theta_2} \del{1+\beta C_A}^2\normSZ{v}1p\Omega\normLt{h}{\I}.
	\end{align}
\end{thm}

\begin{proof} The derivation of \eqref{eq:g_diff_u} is tedious but straightforward, so we skip it. The estimate \eqref{eq:g_diff_u_estimate} follows from Corollary~\ref{cor:g_diff_prelim}.

We turn our focus to proving that $G_v'$ is the Fr\'echet derivative of $G_v$. To this end, we must show that the remainder operator $\Res_{G_v}: \calU \times \Ltwo{\I} \to \bbW^1$, defined as
	\begin{align} \label{eq:frechet_deriv}
	    \Res_{G_v}\sbr{\uop, h} := G_v(\uop + h) - G_v(\uop) - G_v'\del{\uop}h , 
	\end{align}
satisfies for all $\uop \in \calU$
	\[
	     \lim_{\normLt{h}{\I} \to 0} \frac{\norm{\Res_{G_v}\sbr{\uop, h}}_{\bbW^1}}{\normLt{h}{\I}}  = 0.
	\]     
Since we do not have direct access to $\norm{\Res_{G_v}\sbr{\uop,
    h}}_{\bbW^1}$, the strategy of the proof is to first show that
$\Res_{G_v}\sbr{\uop, h}$ satisfies  \eqref{eq:g_diff_general} for
some $F_\Omega\del{\cdot;h} \in \sob1q\Omega^*$ and $F_\Gamma\del{\cdot;h} = 0$,
provided $\normLt{h}{\I}$ is small enough so that $\uop + h \in \calU$; recall
  that $\calU$ is open in $L^2(\I)$. Next, owing to the estimates in Corollary~\ref{cor:g_diff_prelim}, it suffices to check that 
	\[
		\lim_{\normLt{h}{\I} \to 0} \frac{\normSD{F_\Omega\del{\cdot;h}}1q\Omega}{\normLt{h}{\I}}  = 0.
	\]
 
To avoid any ambiguity we adopt the following notation in this proof, 
	\begin{align*}
		\del{\gamma(\uop), y(\uop)} &:= G_v\del{\uop},  &\del{\gamma(\uop+h), y(\uop+h)} &:= G_v(\uop+h) \\
		\del{\gamma_u(\uop)h, y_u(\uop)h} &:= G'_v(\uop)h, &\del{\delta\gamma, \delta y} = \del{\Res_\gamma\sbr{\uop,h},\Res_y\sbr{\uop,h}} &:= \Res_{G_v}\del{\uop, h},
	\end{align*}
whence 
	\begin{align*}
		\delta\gamma &= \gamma\del{\uop + h} - \gamma\del{\uop} - \gamma_u\del\uop h
		&\delta y &= y\del{\uop + h} - y\del{\uop} - y_u\del\uop h.
	\end{align*}

According to the definition \eqref{eq:frechet_deriv} we start by combining \eqref{eq:g_var_pde} for $G_v\del{\uop+h}$ and $G_v\del\uop$ with \eqref{eq:g_diff_u} to obtain for every $\del{\zeta,z}$ in $\sobZ11{\I} \times \sobZ 1 q \Omega$ 
	\begin{align*}
		0 	&= \BG\sbr{\gamma(\uop+h) - \gamma(\uop) - \gamma_u(\uop)h, \zeta} 
				- \DO\sbr{\del{\gamma_u(\uop)h, y_u(\uop)h}, z + E\zeta; \gamma(\uop), y(\uop)}\\
			&\quad + \BO\sbr{y(\uop+h)+v, z+E\zeta; A\sbr{\gamma(\uop+h)}} - \BO\sbr{y(\uop)+v, z+E\zeta; A\sbr{\gamma(\uop)}}.		 
	\end{align*}
Adding and subtracting $\DO\sbr{\del{\gamma(\uop+h) - \gamma(\uop), y(\uop+h)-y(\uop)}, z + E\zeta;\gamma(\uop),y(\uop)}$ to the previous equation and utilizing the definition of $\delta\gamma$ and $\delta y$ above, yields for  every  $\del{\zeta, z}$ in $\sobZ11{\I} \times \sobZ1q\Omega$
	\begin{align*}
		\BG\sbr{\delta \gamma, \zeta} + \DO\sbr{\del{\delta \gamma, \delta y}, z + E\zeta; \gamma(\uop), y(\uop)}
			&= F_\Omega(z + E\zeta;h),
	\end{align*}
where
	\begin{align*}
		F_\Omega(\cdot;h) &= 
			\BO\sbr{y(\uop+h) + v, \cdot; A\sbr{\gamma(\uop)}-A\sbr{\gamma(\uop+h)} } \\
			&\quad +\BO\sbr{y(\uop) + v, \cdot; \Dif A\sbr{\gamma(\uop)}\difdir{\gamma(\uop+h) - \gamma(\uop)}}.
	\end{align*}
The fact that $F_\Omega\del{\cdot;h}$ is in $\sob1q\Omega^*$ follows from the
continuity of $\BO\sbr{w,\cdot;V}$ with $\normSZ{w}1p\Omega$ and
$\normLi{V}\Omega$ bounded uniformly
(c.f. \eqref{eq:A_unif_bound}). Our last step is to add and subtract
$\BO\sbr{y(\uop+h) + v, \cdot; \Dif
  A\sbr{\gamma(\uop)}\difdir{\gamma(\uop+h) - \gamma(\uop)}}$ to
$F_\Omega\del{\cdot;h}$, employ the definition of the remainder $\Res_A$ in \eqref{eq:op_RA} and the Lipschitz estimates \eqref{eq:g_lipschitz_y} and \eqref{eq:g_lipschitz_gamma} to obtain
	\begin{align*}
 		\lim_{\normLt{h}{\I} \to 0} \frac{\normLi{\Res_A\sbr{\gamma(\uop), \gamma(\uop+h) - \gamma(\uop)}}\Omega}{\normLt{h}{\I}}  = 0,
 	\end{align*}
as well as 
	\begin{align*}
		\normSD{F_\Omega\del{\cdot;h}}1q\Omega
			&\leq \del{1+\beta C_A}\normSZ{v}1p\Omega\normLi{\Res_A\sbr{\gamma(\uop), \gamma(\uop+h) - \gamma(\uop)}}\Omega \\
			&\quad+ C_A\normSZ{y(\uop + h) - y(\uop)}1p\Omega\normSZ{\gamma(\uop+h)-\gamma(\uop)}1\infty{\I} \\
			&= \littleo{\normLt{h}{\I}}.
	\end{align*}
This concludes the proof.
\end{proof}

%% file: 04_Control_to_State/g_twice_diff.tex
\subsection{The Second-order Fr{\'e}chet Derivative}
\label{s:g_twice_diff}
The main 
result of this subsection is to show that $G_v(u)$ is twice Fr\'echet differentiable
with respect to $u$. 
We adopt a direct approach in line with subsections \ref{s:g_lipschitz}
and \ref{s:g_differentiable} in favor of the technique based on the implicit function theorem described in \cite[pp. 239-240]{FTroltzsch_2010a}, which would
require nonobvious modifications to account for the nonlinear 
structure of the state equations.
In fact, proceeding as in \thmref{thm:g_lipschitz} we get the following.
\begin{prop}[Lipschitz continuity of $G_v'$]\label{prop:g_twice_diff_u} There exists a constant 
$L_{G'} > 0$, such that for every $u_1, u_2 \in \calU$
	\begin{equation}\label{eq:g_twice_diff_u}
		\sup_{0\neq h \in L^2(\I)}\frac{\norm{G'_v(u_1)h - G'_v(u_2)h}_{\bbW^1}}{\normLt{h}{\I}} \leq L_{G'}\normLt{u_1-u_2}{\I}.
	\end{equation}
\end{prop}

\begin{thm}[the Fr\'echet derivative of $G'_v$] \label{thm:g_twice_diff_u} 
	The control-to-state map $G_v:\calU\to\bbW^1$ 
admits a second-order Fr\'echet derivative  
$G_v'' : \calU \rightarrow \mathcal{L}\del{L^2(\I)\!\times\!
  L^2(\I),\bbW^1}$. Hence for all $\uop \in \calU$ and all
$\del{h_1 , h_2} \in L^2(\I) \times L^2(\I)$, $\del{\gamma, y} :=
G''_v\del{\uop}h_1h_2 \in \bbW^1$ satisfies the linear variational
system \eqref{eq:g_diff_general}, namely for
all $\del{\zeta,z}$ in $\sobZ11{\I} \times \sobZ1q\Omega$
	\begin{align} \label{eq:g_nd_diff}
		 \BG\sbr{\gamma,\zeta} + \DO\sbr{\del{\gamma,y},z+E\zeta;\gop, \yop} 
			&= F_\Omega\del{z+E\zeta;h_1,h_2},
	\end{align}
with $F_\Omega\del{\cdot;h_1,h_2} \in \sob1q\Omega^*$ given by 
	\begin{equation}\begin{aligned} \label{eq:F_Omega}
		F_\Omega\del{\cdot;h_1,h_2}
			&:= -\BO\sbr{y_1,\cdot;\Dif A\sbr\gop\difdir{\gamma_2}}  \\
			&\quad -\BO\sbr{y_2, \cdot; \Dif A\sbr\gop\difdir{\gamma_1}} 
			       -\BO\sbr{\yop + v,\cdot;\Dif^2 \!\!  A\sbr\gop\difdir{\gamma_1,\gamma_2}}, 
	\end{aligned}\end{equation}
and $\del{\gamma_i, y_i} := G_v'\del{\uop}h_i$, for $i=1,2$. Moreover, the following estimates hold 
	\begin{align}
		\label{eq:g_nd_diff_estimate_gamma}
		\normSZ{\gamma}1\infty{\I} 
			&\leq \frac{\alpha^3}{\theta_2^3} C_E C_A\del{1+2\beta C_A}\del{1+\beta C_A}^2\normSZ{v}1p\Omega \normLt{h_1}{\I}\normLt{h_2}{\I},  \\
		\label{eq:g_nd_diff_estimate_y}
		\normSZ{y}1p\Omega
			&\leq \frac{\alpha^2}{\theta_2^3}\beta C_A\del{1+2\beta C_A}\del{1+\beta C_A}\normSZ{v}1p\Omega
				\normLt{h_1}{\I}\normLt{h_2}{\I}. 
	\end{align}
\end{thm}
\begin{proof}
We skip the derivation of \eqref{eq:g_nd_diff} because it is tedious but straightforward. The estimates for $\normSZ{\gamma}1\infty{\I}$ and $\normSZ{y}1p\Omega$ are a consequence of Corollary~\ref{cor:g_diff_prelim} with $F_\Gamma = 0$ after estimating \eqref{eq:F_Omega}, namely
	\begin{align*}
		\normSD{F_\Omega\del{\cdot;h_1,h_2}}1q\Omega
			&\leq C_A\del{ \normSZ{y_1}1p\Omega\normSZ{\gamma_2}1\infty{\I}  
					+ \normSZ{y_2}1p\Omega\normSZ{\gamma_1}1\infty{\I}} \\
			&\quad + C_A\del{1+\beta C_A}\normSZ{v}1p\Omega 
					\normSZ{\gamma_1}1\infty{\I} \normSZ{\gamma_2}1\infty{\I} \\
			&\leq 2\del{\frac{\alpha}{\theta_2}}^2\beta C_A^2\del{1+\beta C_A}\normSZ{v}1p\Omega
				\normLt{h_1}{\I}\normLt{h_2}{\I} \\
			&\quad + \del{\frac{\alpha}{\theta_2}}^2C_A\del{1+\beta C_A}\normSZ{v}1p\Omega \normLt{h_1}{\I}\normLt{h_2}{\I} \\
			&= \del{\frac{\alpha}{\theta_2}}^2C_A\del{1+2\beta C_A}\del{1+\beta C_A}\normSZ{v}1p\Omega \normLt{h_1}{\I}\normLt{h_2}{\I},
	\end{align*}
where we have used \eqref{eq:g_diff_prelim_continuity_gamma}-\eqref{eq:g_diff_prelim_continuity_y} with $F_\Omega = 0$ for $\del{\gamma_i, y_i}$ along with \eqref{eq:u_negative_to_l2}.

	The strategy for showing second-order Fr\'echet differentiability of $G_v$ is the same as in \autoref{thm:g_diff_u}: we first  show that the remainder 
		\begin{align}  \label{eq:sec_ord_res}
			\del{\delta\gamma, \delta y} := G_v'\del{\uop + h_2}h_1 - G_v'\del{\uop}h_1 -
								 G_v''\del{\uop}{h_1h_2},
		\end{align} 
satisfies the linear variational system in \eqref{eq:g_diff_general} for a suitable right-hand side $\delta F_\Omega \in \sob1q\Omega^*$, and prove that 
	\begin{align}\label{eq:res_rhs_little_o}
		\sup_{0\neq h_1 \in L^2(\I)} \frac{\normSD{\delta F_\Omega\del{\cdot;h_1,h_2}}1q\Omega}{\normLt{h_1}{\I}} = \littleo{\normLt{h_2}{\I}} ,
	\end{align}
with $h_1, h_2 \in L^2(\I)$ arbitrary but small enough so that if 
$\uop \in \calU$, then $\uop + h_1, \uop + h_2 \in \calU$.
	
	As a tradeoff between clarity and space we denote $u_i = \uop + h_i$,  
	and 
	\begin{align*}
		\del{\gamma(u_i), y(u_i)} &:= G_v(u_i), \\ 
		\del{\gamma_u\del{u_i} h_j, y_u\del{u_i} h_j} 
			 	&:= G_v'\del{u_i}{h_j}, \\ 
		\del{\gamma_{uu}\del{\uop} h_1h_2, y_{uu}\del{\uop} h_1 h_2}
				&:= G_v''\del{\uop}{h_1h_2}, \\ 
	\end{align*}
	for $i,j = 1,2$, whence
	\begin{align*}
		\delta\gamma &= \gamma_u\del{u_2}h_1 - \gamma_u\del{\uop} h_1 - \gamma_{uu}\del{\uop} h_1 h_2, \\
		\delta y &= y_u\del{u_2}h_1 - y_u\del{\uop} h_1 - y_{uu}\del{\uop} h_1 h_2.
	\end{align*}
	
	According to the definition \eqref{eq:sec_ord_res} we start by combining \eqref{eq:g_diff_u}  for $G_v'\del{u_2}h_1$ and $G_v'\del{\uop} h_1$ with \eqref{eq:g_nd_diff} to obtain for every $\del{\zeta,z}$ in $\sobZ11{\I} \times \sobZ 1 q \Omega$ 
	\begin{align*}
		-F_\Omega\del{z+E\zeta;h_1,h_2}
			&= \BG\sbr{\delta\gamma, \zeta} 
				+ \DO\sbr{G_v'(u_2)h_1, z + E\zeta; G_v(u_2)}
			\\&\quad
				- \DO\sbr{G_v'(\uop)h_1 + G_v''(\uop)h_1h_2, z+ E\zeta; G_v(\uop)},
	\end{align*}
	where $-F_\Omega(\cdot;h_1,h_2) = \sum_{i=1}^3 F_i(\cdot;h_1,h_2)$ is defined in \eqref{eq:F_Omega}. Further manipulation, based on adding to both sides the following two additional terms,
	\begin{align*}
		F_4\del{z + E\zeta;h_1,h_2} &=  \DO\sbr{G_v'(u_2)h_1, z + E\zeta; G_v(\uop)}, \\
		F_5\del{z+ E\zeta;h_1,h_2} &= - \DO\sbr{G_v'(u_2)h_1, z + E\zeta; G_v(u_2)}
	\end{align*}
	leads to 
	\begin{align*}
		 \BG\sbr{\delta\gamma, \zeta} + \DO\sbr{\del{\delta \gamma, \delta y},z + E\zeta;\gamma(\uop), y(\uop)}& = -\delta F_\Omega\del{z+E\zeta;h_1,h_2},
	\end{align*}
where $\delta F_\Omega\del{\cdot;h_1,h_2} = \sum_{i=1}^5 F_i\del{\cdot;h_1,h_2}$ is clearly in $\sob1q\Omega^*$. To create additional cancellations we further decompose $\delta F_\Omega = \sum_{i=1}^{9} T_i$ as follows:
	\begin{align*}
		T_1 
			&= \BO\sbr{y_u(\uop)h_1,\cdot;
					\Dif A\sbr{\gamma(\uop)}\difdir{\gamma(u_2)-\gamma(\uop) - 					\gamma_u(\uop)h_2} }, \\
		T_2 
			&= {\BO\sbr{y_u(u_2)h_1-y_u(\uop)h_1,\cdot;
					\Dif A\sbr{\gamma(\uop)}\difdir{\gamma(u_2)-\gamma(\uop)} }}, \\
		T_3 
			&= \BO\left[{ y_u(u_2)h_1,\cdot;
						A\sbr{\gamma(u_2)} - A\sbr{\gamma(\uop)} 
					}\right.  \left.{
						-\Dif A\sbr{\gamma(\uop)}\difdir{\gamma(u_2)-\gamma(\uop)} 
					}\right], \\
		T_4 
			&= {\BO\sbr{y_u(\uop)h_2,\cdot;
					\Dif A\sbr{\gamma(\uop)}\difdir{\gamma_u(u_2)h_1 
												    - \gamma_u(\uop)h_1} }}, \\
		T_5 
			&= {\BO\sbr{y(u_2)-y(\uop)-y_u(\uop)h_2,\cdot;
			\Dif A\sbr{\gamma(\uop)}\difdir{\gamma_u(u_2)h_1} }}, \\
		T_6 
			&= {\BO\left[{y(u_2)+v,\cdot; 
					\del{\Dif A\sbr{\gamma(u_2)}-\Dif A\sbr{\gamma(\uop)}}\difdir{\gamma_u(u_2)h_1 }
				}\right.} \\ &\qquad\qquad\qquad\qquad\quad {\left.{
					-\Dif^2 \!\!  A\sbr{\gamma(\uop)}\difdir{\gamma_u(u_2)h_1, \gamma(u_2)-\gamma(\uop)} 
				}\right]}, \\
		T_7
			&= {\BO\left[{y(u_2)+v,\cdot;
				}\right.} {\left.{
				\Dif^2 \!\!  A\sbr{\gamma(\uop)}\difdir{\gamma_u(u_2)h_1, 
									    \gamma(u_2)-\gamma(\uop) - \gamma_u(\uop)h_2} 
				}\right]}, \\
		T_8 
			&= {\BO\sbr{y(u_2)+v,\cdot;
			\Dif^2 \!\!  A\sbr{\gamma(\uop)}\difdir{\gamma_u(u_2)h_1-\gamma_u(\uop)h_1, \gamma_u(\uop)h_2} }}, \\
		T_9 
			&= {\BO\sbr{y(u_2)-y(\uop),\cdot;
			\Dif^2 \!\!  A\sbr{\gamma(\uop)}\difdir{\gamma_u(\uop)h_1, \gamma_u(\uop)h_2} }} ,
	\end{align*}
	where $T_i = T_i\del{h_1,h_2}$. 
We now estimate each of these terms separately and show 
	\begin{align}\label{eq:term_little_o}
		\sup_{0\neq h_1 \in L^2(\I)} \frac{\normSD{T_i\del{h_1,h_2}}1q\Omega}{\normLt{h_1}{\I}} = \littleo{\normLt{h_2}{\I}}.
	\end{align}
which obviously imply \eqref{eq:res_rhs_little_o}.

\begin{enumerate}[$\bullet$]
	\item Term $T_1$: Since 
		\[
			\normSD{T_1}1q\Omega
			\leq C_A\normSZ{y_u\del{\uop} h_1}1p\Omega 
				\normSZ{{\gamma\del{u_2} - \gamma\del{\uop} - \gamma_u\del{\uop} h_2}}1\infty{\I},
		\]
	the estimate \eqref{eq:g_diff_u_estimate}, together with 
		\[
			\normSZ{\gamma\del{u_2} - \gamma\del{\uop} - \gamma_u\del{\uop} h_2}{1}{\infty}{\I} = \littleo{\normLt{h_2}{\I}},
		\]
	implies \eqref{eq:term_little_o}.
	
	\item Term $T_2$: Since
	 	\[
		 	\normSD{T_2}1q\Omega 
				\leq C_A\normSZ{y_u\del{u_2}h_1 - y_u\del{\uop}h_1}1p\Omega\normSZ{\gamma\del{u_2}-\gamma\del{\uop}}1\infty{\I},
		\]
	it suffices to recall the Lipschitz properties \eqref{eq:g_lipschitz} of $G_v$, and \eqref{eq:g_twice_diff_u} of $G_v'$ to deduce \eqref{eq:term_little_o}.
	
	\item Term $T_3$: Invoking the Fr\'echet differentiability \eqref{eq:op_RA} of $A$,
	and the Lipschitz property \eqref{eq:g_lipschitz_gamma} of $\gamma(\uop)$ we infer that
		\begin{multline*}
            		\normLi{A\sbr{\gamma\del{u_2}} - A\sbr{\gamma\del{\uop}} - \Dif A\sbr{\gamma\del{\uop}}\difdir{\gamma\del{u_2} - \gamma\del{\uop}}}\Omega  \\                      		= \littleo{\normSZ{\gamma\del{u_2} - \gamma\del{\uop}}1\infty{\I}} = \littleo{\normLt{h_2}{\I}}.
		\end{multline*}
		This, in conjuction with $\normSZ{y_u\del{u_2} h_1}1p\Omega \lesssim\normLt{h_1}{\I}$, yields \eqref{eq:term_little_o}.
	
	\item Term $T_4$: In view of the Lipschitz property \eqref{eq:g_twice_diff_u} of $G_v'$ 
		\[
			\normSZ{\gamma_u\del{u_2} h_1 - \gamma_u\del{\uop}h_1}1\infty{\I}
				\lesssim \normLt{h_1}{\I}\normLt{h_2}{\I},
		\]
	property \eqref{eq:term_little_o} follows from $\normSZ{y_u\del{\uop} h_2}1p\Omega \lesssim \normLt{h_2}{\I}$.

	\item Term $T_5$: Since $y(u)$ is Fr\'echet differentiable according to \autoref{thm:g_diff_u}, namely 
		\[
	        \normSZ{y\del{u_2} - y\del{\uop} - y_u\del{\uop} h_2}1p\Omega 
	      		= \littleo{\normLt{h_2}{\I}} ,
		\]
	the bound \eqref{eq:term_little_o} is a consequence of $\normSZ{\gamma_u\del{u_2} h_1}1\infty{\I} \lesssim \normLt{h_1}{\I}$.
	
	\item Term $T_6$: We recall the second-order Fr\'echet differentiability of the matrix $A$ with respect to $\gamma$, namely \eqref{eq:op_RDA}, and the Lipschitz continuity \eqref{eq:g_lipschitz} of $G_v$, to write
	 	\begin{multline*}
	 		\Bigg\lVert 
				\Dif A\sbr{\gamma\del{u_2}}\difdir{\gamma_u(u_2)h_1} - \Dif A\sbr{\gamma\del{\uop}}\difdir{\gamma_u(u_2)h_1} \\
					- \Dif^2 \!\! A\sbr{\gamma\del{\uop}}\difdir{\gamma_u\del{u_2}h_1, \gamma\del{u_2}-\gamma\del{\uop}}
			\Bigg\rVert_\Linf{\I} \\
			 = \normLt{h_1}{\I}\littleo{\normSZ{\gamma\del{u_2} - \gamma\del{\uop}}1\infty{\I}}
				   = \normLt{h_1}{\I} \littleo{\normLt{h_2}{\I}}.
		\end{multline*} 
	Since $\normSZ{y(u_2)+ v}1p\Omega \lesssim \normSZ{v}1p\Omega$, this implies \eqref{eq:term_little_o}.

	\item Term $T_7$: We proceed as with $T_6$, now appealing to \eqref{eq:A_unif_bound} and the Fr\'echet differentiability of $\gamma$ at $\uop$ (\autoref{thm:g_diff_u}), to obtain
		\begin{align*}
		\normSZ{\gamma_u(u_2)h_1}1\infty{\I} 
			\normSZ{\gamma(u_2) - \gamma(\uop) - \gamma_u(\uop)h_2}1\infty{\I}
				&= \normLt{h_1}{\I} \littleo{\normLt{h_2}{\I}},
		\end{align*}
	whence \eqref{eq:term_little_o}.
	
	\item Term $T_8$: We employ the Lipschitz property \eqref{eq:g_twice_diff_u} of $G_v'$ to write 
		\[
			\normSZ{\gamma_u(u_2)h_1 - \gamma_u(\uop)h_1}1\infty{\I}				\lesssim \normLt{h_1}{\I} \normLt{h_2}{\I}.
		\]
	The desired bound \eqref{eq:term_little_o} follows from 
	$\normSZ{\gamma_u\del{\uop} h_2}1\infty{\I} \lesssim \normLt{h_2}{\I}$.
	
	\item Term $T_9$: We use the Lipschitz property \eqref{eq:g_lipschitz} of $G_v$,
		\[
			\normSZ{y(u_2)-y(\uop)}1p\Omega \lesssim \normLt{h_2}{\I},
		\]
	 together with $\normSZ{\gamma_u\del{\uop} h_j}1\infty{\I} \lesssim \normLt{h_j}{\I}$ to deduce \eqref{eq:term_little_o}.
\end{enumerate}
Altogether, this concludes the proof.
\end{proof}

We next state without proof that the second-order Fr\'echet 
derivative $G_v''$ of the control-to-state map $G_v$
is Lipschitz continuous; the proof is similar to that of
\thmref{thm:g_lipschitz} and is thus omitted. We need this result
later in Corollary \ref{cor:quad_growth_cond}.
\begin{prop}[Lipschitz continuity of $G_v''$]\label{prop:guu_Lipschitz} There exists a constant 
$L_{G''} > 0$, such that for every $u_1, u_2 \in \calU$
	\begin{equation}\label{eq:guu_Lipschitz}
		\sup_{0\neq h_1,h_2 \in L^2(\I)}\frac{\norm{G''_v(u_1)h_1 h_2 - G''_v(u_2)h_1h_2}_{\bbW^1}}{\normLt{h_1}{\I}\normLt{h_2}{\I}} \leq L_{G''}\normLt{u_1-u_2}{\I}.
	\end{equation}
\end{prop}

%% file: 05_Reduced_Cost/oc_reduced_cost.tex
\section{Optimal Control} \label{s:oc_rc}
Let us summarize what we have accomplished so far. We have formally derived the 
first-order necessary optimality conditions in \autoref{s:oc_lagrangian}. If 
$G_v$ denotes the control-to-state map, we have proved in \autoref{s:g_map} that
$G_v$ is well posed, i.e., there exists a unique weak solution to the state equations 
\eqref{eq:lin_st_const} for every $u\in \calU$ in \eqref{eq:open_ball}, and $v$ satisfying \eqref{eq:v_invariant}
and \eqref{eq:v_contraction}. As a crucial step forward we 
have shown that 
$G_v$ is twice Fr\'echet differentiable on $\calU$. 

This background work puts us in the position to show the existence and (local)
uniqueness of the optimal control $u$ solving the OC-FBP in \eqref{eq:lin_cost}-\eqref{eq:lin_st_const}.
We will achieve this result in three stages. We first show the existence of $u$ in 
Theorem~\ref{thm:oc_rc_existence} of \autoref{s:oc_rc_existence}. We next derive 
the first-order necessary optimality conditions and show the existence and uniqueness of the solution to the 
adjoint equations in \autoref{s:oc_rc_first}. Finally in \autoref{s:oc_rc_second} 
we end this voyage by proving the second-order sufficient conditions for the control $u$. 

%% file: 05_Reduced_Cost/oc_rc_existence.tex
\subsection{Existence of Optimal Control}
\label{s:oc_rc_existence}
In order to show the existence of a solution to our optimal control problem we
first rewrite
the cost functional $\calJ: \bbW^1 \times\Uad\to\bbR$ 
from \eqref{eq:lin_cost} in its reduced form. This is accomplished by utilizing the control-to-state map
$G_v$ from Section~\ref{s:g_map} as follows:
	\begin{align*}
		\calJ(u) &:= \calJ(G_v(u),u) = \calJ(\gamma,y,u) = \calJ_1(G_v(u)) + \calJ_2(u),
	\end{align*}
	with
	\begin{align*}
		\calJ_1(G_v(u)) &:= \frac{1}{2}\normLt{\gamma-\gamma_d}{\I}^2 
	        + \frac{1}{2}\normLt{(y+v-y_d)\sqrt{J_\gamma}}\Omega^2, 
		& \calJ_2(u) &:= \frac{\lambda}{2}\normLt{u}{\I}^2 .
	\end{align*}
Thus, after recalling that $\Uad$ is a closed subset of $\calU$, we obtain that
	\begin{align}\label{eq:lin_cost_reduced}
		\min_{u \in \Uad} \calJ(u)
	\end{align}
is an equivalent minimization problem to \eqref{eq:lin_cost}.

\begin{thm}[existence of optimal control] \label{thm:oc_rc_existence}
For every $v$ satisfying \eqref{eq:v_invariant} and \eqref{eq:v_contraction} there exists an optimal control 
$\uop\in \Uad$ minimizing
the cost functional \eqref{eq:lin_cost} with optimal state 
$(\gop,\yop) \in \left( \sob{1+1/q}{p}{{\I}} \cap \sobZ{1}{\infty}{{\I}} \right) \times \sobZ{1}{p}{\Omega}$ which
solves the free boundary problem \eqref{eq:lin_st_const} and satisfies the state 
constraint \eqref{eq:state_const}.
\end{thm}
\begin{proof}
In order to show the existence of an optimal control we use the direct method of the calculus of variations. We first 
note that the cost functional $\calJ$ in \eqref{eq:lin_cost_reduced}
is bounded below by zero, whence $j = \inf_{u \in \Uad} \calJ(u)$ is finite. We thus construct a minimizing sequence
$\{ u_n \}_{n\in \mathbb{N}}$ such that
	\[
  		j = \lim_{n\rightarrow \infty} \calJ(u_n).
	\]  
By Definition~\ref{defn:const_set}, $\Uad$ is nonempty, closed, bounded
and convex in $\Ltwo{\I}$, thus weakly sequentially compact. Consequently,
we can extract a weakly convergent subsequence $\{ u_{n_k} \}_{k \in \mathbb{N}} \subset \Ltwo{\I}$, i.e.
	\[
     		u_{n_k} \weakto \uop  \quad \env{in } \Ltwo{\I}, \qquad \uop \in \Uad.
	\]
Here $\uop$ is our optimal control candidate. 

Henceforth, we drop the subindex $k$ when extracting subsequences. 
According to Corollary~\ref{cor:g_operator} and \eqref{eq:gamma_enhanced_regularity},
we let $G_v(u_n) = (\gamma_n,y_n)\in \left( \sob{1+1/q}{p}{{\I}} \cap \sobZ{1}{\infty}{{\I}} \right) \times \sobZ{1}{p}{\Omega}$ denote the unique state corresponding to $u_n$, thereby  solving the free boundary problem 
\eqref{eq:lin_st_const} and satisfying the state constraint \eqref{eq:state_const}. Since 
$\sob{1+1/q}{p}{{\I}} \cap \sobZ{1}{\infty}{{\I}} $ is compactly
embedded into $\sobZ{1}{\infty}{{\I}}$ for $p>2$,
the Rellich-Kondrachov theorem yields a strongly convergent subsequence 
$\{ \gamma_{n} \}_{n\in \mathbb{N}} \subset \sobZ{1}{\infty}{{\I}}$, i.e. 
	\[
     		\gamma_n \to  \gop \quad \env{in } \sobZ{1}{\infty}{{\I}}, \qquad \env{and} \qquad
     		y_n \rightharpoonup  \yop \quad \env{in } \sobZ{1}{p}{\Omega} .
	\]
Note that the limit pair $\del{\gop, \yop}$ is the state corresponding to the control $\uop$. This results from replacing 
$(\gamma,y)$ with $(\gamma_n,y_n)$
in the variational equation \eqref{eq:g_var_pde} taking the limit, and making use of the embedding $L^2({\I}) \subset \sob{-1}{\infty}{{\I}}$.

Finally, using the fact that $\calJ_2\del{u}$ is continuous in $L^2$ and convex, together with the strong convergence 
$\del{\gamma_n, y_n} \rightarrow \del{\gop, \yop}$ in $L^\infty\del{{\I}}\times L^\infty\del{\Omega} $, 
again due to Rellich-Kondrachov theorem, it follows that $\calJ$ is weakly lower semicontinuous, whence
	\begin{align*}
    		\inf_{u\in\Uad} \calJ\del{u} 
			= \liminf_{n\to\infty} \del{\calJ_1(G_v(u_n)) + \calJ_2(u_n)}
			\ge \calJ_1(G_v(\uop)) + \calJ_2(\uop) 
		      	= \calJ\del{\uop}.
	\end{align*}
This concludes the proof.
\end{proof}

%% file: 05_Reduced_Cost/oc_rc_first.tex
\subsection{First-order Necessary Condition}
\label{s:oc_rc_first}
We start with a classical result \cite{FTroltzsch_2010a}. 
\begin{lem}[variational inequality] 
If $\uop \in \Uad$ denotes an optimal control, given by Theorem~\ref{thm:oc_rc_existence}, then the 
first order necessary optimality condition satisfied by $\uop$ is
\begin{align}  \label{eq:var_ineq}
       \pairLt{\mathcal{J}'(\uop), u - \uop}{\I}  &\ge 0  \quad  \forall u \in \Uad .
\end{align}
\end{lem}
	We will show that the variational inequality \eqref{eq:var_ineq} is the same as 
\eqref{eq:lin_uop_var_ineq} as well as prove that \eqref{eq:lin_rop} and \eqref{eq:lin_sop}
are the correct adjoint equations. This furnishes a rigorous
derivation of the formal results of section \ref{s:oc_lagrangian}.

Since the set $\Uad$ defined in
\eqref{eq:v_u_const_set} is closed, we need to deal with a suitable 
set of admissible directions.
\begin{defn}[admissible directions] \label{def:u_convex_cone}
Given $u \in \Uad$, the convex cone $\cone{u}$ comprises all directions 
$h \in \Ltwo{{\I}}$ such that $u+th \in \Uad$ for some $t>0$, i.e.,
	\[
    	\cone{u} := \set{h \in \Ltwo{{\I}}:\; u+th \in \mathcal
          U_{ad}, t>0}.
	\]    
\end{defn}
\begin{thm}[first-order conditions] \label{thm:oc_rc_first}
If $\uop \in \Uad$ denotes an optimal control of OC-FBP, then the first-order necessary optimality conditions are 
given by \eqref{eq:lin_rop}, \eqref{eq:lin_sop} and \eqref{eq:lin_uop_var_ineq}. 
\end{thm}
\begin{proof}
We can infer that $\mathcal J$ is Fr\'echet differentiable by recalling from 
Theorems~\ref{thm:g_diff_u} and \ref{thm:g_twice_diff_u} that $G_v$ is twice differentiable and 
that $\calJ_1$ is quadratic. In fact, the Fr\'echet derivative of $\calJ$ in 
\eqref{eq:lin_cost_reduced} at $\uop$ in a direction $h \in \cone{\uop}$ is
\begin{align*}
	\calJ'(\uop) h &= \calJ_1'(G_v(\uop))G_v'(\uop) h + \calJ_2'(\uop) h  
                     = \calJ_1'(G_v(\uop)) (\gamma_u(\uop)h, y_u(\uop)h) 
                     + \calJ_2'(\uop) h ,
\end{align*}
where $G_v'(\uop) h = (\gamma_u(\uop)h, y_u(\uop)h)$ 
satisfies \eqref{eq:g_diff_u} and 
	\begin{equation}\begin{aligned} \label{eq:red_cost_first_deriv}
		\calJ'(\uop) h
			&=  \pairLt{   \del{\yop+v-y_d}\del{1+\gop}, y_u(\uop)h}{\Omega} \\ 
            &\quad +  \Big\langle{\gop-\gamma_d+
              \frac{1}{2}\int_0^1\abs{\yop+v-y_d}^2 \dif x_2,\gamma_u(\uop)h}\Big\rangle_{L^2(\I)\times L^2(\I)} \\
	        &\quad + \lambda \pairLt{\uop, h}{\I} . 
	\end{aligned}\end{equation}
Introducing the adjoint states $(\rop,\sop) \in \sob{1}{q}{\Omega} \times \sobZ{1}{1}{{\I}}$, which satisfy the system 
\eqref{eq:lin_rop}-\eqref{eq:lin_sop} in weak form, and noting that $h \in L^2({\I})$, we obtain
\begin{align*}
	\mathcal{J}'(\uop) h 
		&= \BG\sbr{\gamma_u(\uop)h,\sop} + \DO\sbr{\del{\gamma_u(\uop)h, y_u(\uop)h},\rop; \gop, \yop} + \lambda \pairLt{\uop, h}{\I}. 
\end{align*}
Utilizing \eqref{eq:g_diff_u} with $\zeta = \sop$ and $z = \rop$, we arrive at
\begin{align*}
	\mathcal{J}'(\uop) h	
    		&= \pairLt{\sop + \lambda \uop, h}{\I}
			+ \DO\sbr{\del{\gamma_u(\uop)h, y_u(\uop)h},\rop - E \sop;\gop, \yop}.
\end{align*}
Since the Dirichlet 
condition $\rop\restriction_\Gamma = \sop$ implies $\rop - E \sop \in \sobZ{1}{q}{\Omega}$,
\eqref{eq:g_diff_u}  with $\zeta = 0$ and $z \in \sobZ1q\Omega$ yields 
$\DO\sbr{\del{\gamma_u(\uop)h}, y_u(\uop)h, z; \gop, \yop} = 0$, whence
\begin{align*}	
     \mathcal{J}'(\uop) h				 
        		&= \pairLt{\sop + \lambda \uop, h}{\I}.
\end{align*}
In view of \eqref{eq:var_ineq}, this coincides with \eqref{eq:lin_uop_var_ineq} for $h = u - \uop$ admissible. 
\end{proof}

%% file: 05_Reduced_Cost/adjoint_wellposed.tex
\subsubsection{Well-posedness of the Adjoint System}
\label{s:oc_rc_adjoint}
Before we dwell upon the second-order sufficient optimality conditions we put together the last 
piece of the puzzle: the well-posedness of the adjoint system 
\eqref{eq:lin_rop} and \eqref{eq:lin_sop}. This will be done using a contraction argument in 
Banach spaces, assuming that we have a solution $\del{\gop,\yop} \in \bbB_v$ 
to the state equations in \eqref{eq:lin_st_const} satisfying \propref{prop:b_infsup} and \lemref{lem:t_maps_weps}.

Let $\bbV := \set{r \in \sob{1}{q}{\Omega}\,:\, r\restriction_\Gamma \in \sobZ{1}{1}{{\Gamma}},\, r\restriction_\Sigma = 0}$, 
and let the operator
$T_1:\bbV \to \sobZ 1 1 {\I}$ be defined as $\stl = T_1(r)$ where $\stl$ satisfies for every $\zeta \in \sobZ 1 \infty {\I}$ 
\begin{align}\label{eq:t1_adj_var}
	\BG\sbr{\zeta,\stl} &= \pair{\zeta, {\bf f}\sbr{\gop,\yop,r}}_{\sobZ{1}{\infty}{{\I}},
	\sob{-1}{1}{{\I}}},
\end{align}
with 
(recall the identification of $\Gamma$ with $\I$ as well as that 
of Sobolev spaces on them)
	\begin{align*}
	     &\pair{\zeta,\bf f}_{\sobZ 1 \infty {\I},\sob{-1}{1}{{\I}}} \\
	     &  \quad\quad\quad\quad := \pair{\zeta, f_0\sbr{x_1;\gop,\yop,r}}_{\Linf{\I}, \Lone{\I}} 
				+ \pair{\totalD_{x_1}\zeta,f_1\sbr{x_1;\gop,\yop,r}}_{\Linf{\I},\Lone{\I}}  
	\end{align*}			
and			
\begin{align}	
    	f_0[\cdot;\gop, \yop, r] &:=  \gop - \gamma_d +
    	 \frac{1}{2}\int_0^1\abs{\yop+v-y_d}^2 \dif x_2 
    	 - \int_0^1 A_1\sbr\gop \grad \del{\yop+v} \cdot \grad r\dif x_2, \nonumber \\
    	f_1[\cdot;\gop, \yop, r] &:=  - \int_0^1 A_2\sbr\gop \grad \del{\yop+v} \cdot \grad r\dif x_2 	\label{eq:f0_f1_adj},
	\end{align}
where $A_1\sbr\gop,A_2\sbr\gop$ are defined in \eqref{eq:op_DA}. 
Given $\stl \in \sobZ11{\I}$, let $T_2:\sobZ 1 1 {\I} \to \bbV$ 
be the operator defined as 
$\rtl = T_2(\stl) \in E \stl + \sobZ 1 q \Omega$ satisfying
\begin{equation}\label{eq:t2_adj_var}
	\BO\sbr{z,\rtl;A\sbr{\gop}} =  \pairLpq{z,\del{\yop +v - y_d}
          \del{1+\gop}}\Omega  
	  \qquad \forall z \in \sobZ{1}{p}{\Omega} .
\end{equation}

\begin{lem}[ranges of $T_1$ and $T_2$] \label{lem:t_adj_maps_back}
Let $T_1$, $T_2$ be defined in \eqref{eq:t1_adj_var} and \eqref{eq:t2_adj_var}. If
$\del{\gop,\yop} \in \bbB_v$, the ball
defined in \eqref{eq:W1_convex_set}, then
\begin{enumerate}[(i)]
\item\label{item:t_adj_maps_back_i}
	for all $r \in \bbV$, 
        the solution $\stl = T_1(r)$ to \eqref{eq:t1_adj_var} satisfies
	\begin{align*}
		\normSZ{\stl}{1}{1}{{\I}} 
			&\le  \alpha \Big( \norm{\gop - \gamma_d}_\Lone{{\I}} \\
                        &\quad + \frac{1}{2}\normLt{\yop+v-y_d}\Omega^2  
			    +  C_A\del{1+\beta C_A}\normSZ{v}1p\Omega\normSZ{r}{1}{q}{\Omega}  \Big) ;
    \end{align*}   
\item\label{item:t_adj_maps_back_ii} 
	for all $\stl \in \sobZ{1}{1}{{\I}}$, the solution $\rtl = T_2\del{\stl}$ to \eqref{eq:t2_adj_var} satisfies
    \[
    	\normSZ{\rtl}{1}{q}{\Omega} 
    		\le  2\beta\del{ \norm{\yop + v - y_d}_{\Lq\Omega}
                    + C_E C_A \normSZ{\stl}{1}{1}{{\I}}}  .
    \] 
\end{enumerate}
\end{lem}
\begin{proof}
Using \eqref{eq:b2_infsup} of Proposition~\ref{prop:b_infsup}, and
applying Banach-Ne\v{c}as theorem 
\cite{AErn_JLGuermond_2004a}, there exists a unique solution $\rtl$ to \eqref{eq:t2_adj_var}. 
Estimate (\ref{item:t_adj_maps_back_ii}) follows from \eqref{eq:b2_infsup} and \eqref{eq:A_unif_bound}, as well as 
the Poincar\'e inequality $\norm{z}_{L^p(\Omega)} \le \normSZ{z}{1}{p}{\Omega}$ for the unit square. 

In order to show the existence of solution to \eqref{eq:t1_adj_var}  we note that we are looking for an absolutely 
continuous function $\widetilde{s}$ on the interval $\I=\intoo{0,1}$ with zero Dirichlet values. Therefore, by the characterization of 
such functions in $\bbR$ \cite[Theorem 5.14]{HRoyden_1988}, there exists a $g \in \Lone{\I}$
such that
	\[
		\stl(x_1) = \int_0^{x_1} \Big( g(t)- \int_0^1 g(\tau) \dif \tau \Big) \dif t 
			\qquad \forall x_1 \in \I
	\]
because $\stl(0) = \stl(1) = 0$. The variational equation
\eqref{eq:t1_adj_var} is satisfied by $\stl$ with
	\[
        g(t) = \frac{1}{\kappa} \del{f_1[t;\gop, \yop, r] -
            \int_0^t f_0[\tau; \gop, \yop, r]\dif\tau} 
	\]
and $f_0, f_1$ defined in \eqref{eq:f0_f1_adj}. 

It remains to check that $\int_0^t f_0 d \tau$ 
and $f_1$ are in $\Lone{\I}$. This follows by applying Fubini's
theorem and H\"older's inequality. We consider first $f_0$:
\begin{align*}
	\int_0^1 \abs{f_0[x_1;\cdot]} & \dif x_1 
		 \leq  \norm{\gop - \gamma_d}_{\Lone{\I}} 
		 + \frac{1}{2}\int_\Omega \abs{\yop+v-y_d}^2 \dif x 
			+ \int_\Omega \abs{A_1\sbr{\gop}\grad\del{\yop+v}\cdot\grad r} \dif x \\
		&\leq  \norm{\gop - \gamma_d}_{\Lone{\I}} 
		 + \frac{1}{2}\normLt{\yop+v-y_d}\Omega^2
			+  C_A \del{1+\beta C_A}\normSZ{v}1p\Omega\normSZ{r}{1}{q}{\Omega} ,
\end{align*}
because $\normSZ{\yop+v}{1}{p}{\Omega} \le \del{1+\beta C_A} \normSZ{v}{1}{p}{\Omega}$.
Similarly, we obtain an $L^1$ estimate for $f_1$
\begin{align*}
	\int_0^1 \abs{f_1(x_1)}\dif x_1
		&\leq \int_\Omega \abs{A_2\sbr\gop\grad \del{\yop+v} \cdot \grad r}  \dif x 
		 \le  C_A \del{1+\beta C_A} \normSZ{v}{1}{p}{\Omega} \normSZ{r}{1}{q}{\Omega} .
\end{align*}
This implies the existence of a solution $\widetilde{s} \in \sobZ11{\I}$ to \eqref{eq:t1_adj_var}. The uniqueness of $\widetilde{s}$ and a priori bound in $(i)$ follow from
the estimate in \eqref{eq:b1_infsup_b}.
\end{proof}

\begin{thm}[existence of the adjoint system] \label{thm:t_adj_contraction} 
Under the assumptions of \linebreak
Lemma~\ref{lem:t_adj_maps_back}, the operator $T = T_2 \circ T_1 : \bbV \rightarrow \bbV$ is a
contraction with constant $1-\theta_2$.
\end{thm}
\begin{proof} The proof is similar to the one in \thmref{thm:t_contraction}, therefore we will
be brief. Consider $r_1,\, r_2 \in \bbV$ such that $r_1 \neq r_2$ and let $\stl_i = T_1(r_i), \rtl_i = T_2(\stl_i)$, where
$\stl_i,\rtl_i$ solve
\eqref{eq:t1_adj_var} and \eqref{eq:t2_adj_var} for $i=1, 2$. Then
\propref{prop:b_infsup} (\ref{item:b_infsup_i}) and Lemma~\ref{lem:t_adj_maps_back}
(\ref{item:t_adj_maps_back_i}) imply
	\begin{align*} 
		\normSZ{\stl_1 - \stl_2}{1}{1}{{\I}} 
			&\leq \alpha \sup_{\normSZ{\zeta}{1}{\infty}{{\I}}=1} \BG\sbr{\zeta,\stl_1 - \stl_2} 
				\nonumber \\
			&= \alpha \sup_{\normSZ{\zeta}{1}{\infty}{{\I}}=1} 
				\pair{\zeta,{\bf f}\sbr{\gop,\yop,r_1} - {\bf f}\sbr{\gop, \yop, r_2}} 
					\nonumber \\
			&\leq  \alpha C_A \del{1+\beta C_A}\normSZ{v}1p\Omega\normSZ{r_1 - r_2}{1}{q}{\Omega} . 
	\end{align*}
         In addition, writing $\rtl_i=\tilde\ell_i+E\stl_i$ with 
           $\tilde\ell_i\in\sobZ 1 q{\Omega}$ for $i=1,2$, we see that
	\begin{align*}  
		\normSZ{\rtl_1 - \rtl_2}{1}{q}{\Omega} 
			&= \normSZ{\tilde \ell_1 + E\stl_1 - \tilde \ell_2 - E\stl_2}{1}{q}{\Omega} 
			     \leq C_E\normSZ{\stl_1 - \stl_2}{1}{1}{{\I}} + \normSZ{\tilde \ell_1 - \tilde \ell_2}{1}{q}{\Omega}.
	\end{align*}
	Since \propref{prop:b_infsup} (\ref{item:b_infsup_ii}) implies
	\begin{align*}		
	        \normSZ{\tilde \ell_1 - \tilde \ell_2}{1}{q}{\Omega} 
			&\le \beta\sup_{\normSZ{z}{1}{p}{\Omega}=1} 
					\BO\sbr{z,\tilde \ell_1 - \tilde \ell_2;A\sbr{\gop}}  \\
			&=   \beta \sup_{\normSZ{z}{1}{p}{\Omega}=1}
					\BO\sbr{z,E\stl_2 - E\stl_1;A\sbr\gop} ,
        \end{align*}		
        we deduce that 			
        \begin{align*}  
             \normSZ{\rtl_1 - \rtl_2}{1}{q}{\Omega} 
			&\leq C_E\del{1 + \beta C_A}\normSZ{\stl_1 - \stl_2}11{\I} \nonumber \\
			&\leq  \alpha C_EC_A\del{1+\beta C_A}^2\normSZ{v}1p\Omega\normSZ{r_1 - r_2}{1}{q}{\Omega}.
		\end{align*}
	Invoking \eqref{eq:v_contraction} we obtain
		\begin{align*}
			\normSZ{\rtl_1 - \rtl_2}{1}{q}{\Omega} 
			&\leq \del{1-\theta_2}\normSZ{r_1 - r_2}{1}{q}{\Omega}.
		\end{align*}
	Therefore $T = T_2 \circ T_1 : \bbV \rightarrow \bbV$ is a contraction in $\bbV$. 
\end{proof}

%% file: 05_Reduced_Cost/oc_rc_second.tex
\subsection{Second-order Sufficient Condition}
\label{s:oc_rc_second}
The final step is to prove the second-order sufficient condition for the optimal control $\uop$ found earlier,
which in turn guarantees that  $\uop$ is locally unique. This imposes an additional condition on 
$\normSZ{v}{1}{p}{\Omega}$ depending on the parameter $\theta_3$ given by 
\begin{equation}
\begin{aligned} \label{eq:theta_second_order_suff}
 \theta_3 = 
 \frac{\theta_2^2}{\alpha^2\Lambda_1}  
   \Bigg(
	     \frac{C_A\Lambda_2}{\theta_2}
	    \bigg( \alpha C_E\Lambda_1
            \Big( \omega_1+ \frac{1}{2} \omega_2^2
			     \Big)       
	      + 2\beta \omega_2
		 \bigg) 
	     +2 \Lambda_1^2  \omega_2
	\Bigg)^{-1},
\end{aligned}
\end{equation}
where 
$\Lambda_1 := 1 + \beta C_A, \Lambda_2 := 1 + 2\beta C_A, \omega_1 := 1
+ \normLt{\gamma_d}{\I}$ and $\omega_2 :=  \frac{1-\theta_1}{\alpha C_E C_A} 
+ \normLt{y_d}{\Omega}$.

\begin{thm}[second-order sufficient conditions] \label{thm:second_order_suff}
If $\theta_1$, $\theta_2$ satisfy \eqref{eq:v_invariant}, \eqref{eq:v_contraction}, and in addition
\begin{align}\label{eq:v_second_order_suff}
     \normSZ{v}{1}{p}{\Omega} \le \frac{\theta_3 \lambda}{2} ,
\end{align}
then
\begin{align}   \label{eq:second_order_suff}
	\calJ''(\uop) h^2 &\ge \frac{\lambda}{2} \normLt{h}{\I}^2 \qquad \forall h \in \cone\uop .
\end{align}
\end{thm}
\begin{proof}
Since $G_v$ is twice Fr\'echet differentiable, according to Theorem~\ref{thm:g_twice_diff_u},
we write the second-order Fr\'echet derivative of $\calJ$ from \eqref{eq:lin_cost_reduced} at $\uop$ in the direction $h^2$ as
\begin{align*}  
\calJ''(\uop)h^2 = \calJ_1''(G_v(\uop)) (G_v'(\uop)h)^2 
	+ \calJ_1'(G_v(\uop))G_v''(\uop) h^2 + \calJ_2''(\uop) h^2.
\end{align*}
Recalling that $G_v'(\uop)h = (\gamma_u(\uop)h,y_u(\uop)h)$
and $G_v''(\uop)h^2 = (\gamma_{uu}(\uop)h^2,y_{uu}(\uop)h^2)$, we can
rewrite the three terms on the right-hand side as follows:
\begin{align*}
      \calJ_1''(G_v (\uop))  (G_v'(\uop)h)^2  
     & = \normLt{\gamma_u(\uop)h}{\I}^2  \\
     &\quad + \int_\Omega \del{y_u(\uop)h}^2 \del{1+\gop}  \dif x \\
     &\quad + 2 \int_0^1 \del{\int_0^1 \del{\yop+v-y_d}
       y_u(\uop)h \dif x_2 } \gamma_u(\uop) h \dif x_1 ;
     \\
     \calJ_1'(G_v  (\uop))  (G_v''(\uop)h^2) 
      &=  \int_0^1(\gop - \gamma_d) \gamma_{uu}(\uop)h^2 \dif  x_1 \\
      &\quad + \int_0^1 \del{\frac{1}{2} \int_0^1 \abs{\yop+v-y_d}^2 \dif x_2 } 
           \gamma_{uu}(\uop) h^2 \dif x_1   \\
      &\quad +  \int_{\Omega} \del{\yop + v - y_d} \del{1+\gop} 
           y_{uu}(\uop)h^2 \dif x;  \\
     \calJ_2'' (\uop) h^2  &= \lambda \normLt{h}{\I}^2 .
\end{align*}                         
This yields
\begin{align*}		
         \calJ''(\uop)h^2  	
		&\ge \lambda \normLt{h}{\I}^2 + \normLt{\gamma_u(\uop)h}{\I}^2 
			+  \normLt{y_u(\uop)h \sqrt{1+\gop}}\Omega^2  \\
		&\quad- \del{ \normLt{\gop - \gamma_d}{\I} 
		               +\frac{1}{2} \normLt{\yop+v-y_d}\Omega^2 } \normSZ{\gamma_{uu}(\uop)h^2}1\infty{\I}  \\
	    & \quad -2 \normLt{\yop+v-y_d}\Omega  \normSZ{y_{uu}(\uop)h^2}1p\Omega  \\
	    & \quad -2 \normLt{\yop+v-y_d}\Omega \normSZ{y_u(\uop)h}1p\Omega \normSZ{\gamma_u(\uop)h}1\infty{\I}, 
\end{align*}
because of Poincar\'e inequalities 
$\normLt{z}{\Omega} \le \normSZ{z}{1}{2}\Omega \le \normSZ{z}{1}{p}\Omega$
for $z\in \sobZ 1 p{\Omega}$ and 
$\normLt{\zeta}{{\I}} \le \normLi{\zeta}{\I} \le \normSZ{\zeta}{1}{\infty}{\I}$
for $\zeta\in\sobZ{1}{\infty}{\I}$ on the unit domains $\Omega$ and $\I$.

To estimate the norms above, we first observe that $\del{\gop, \yop}\in \bbB_v$
implies
\begin{equation*}
    	\normLt{\gop-\gamma_d}{\I}  
		\le \normLt{\gop}{\I} + \normLt{\gamma_d}{\I} \le 1 
                + \normLt{\gamma_d}\Omega = \omega_1,
\end{equation*}
and, since $\normLt{\yop+v}{\Omega} \le
\normSZ{\yop+v}{1}{p}{\Omega}\le \del{1+\beta C_A} \normSZ{v}1p\Omega$,
\begin{equation*}
    	\normLt{\yop+v-y_d}\Omega 
                \le \normLt{\yop+v}{\Omega} + \normLt{y_d}\Omega
                \le \frac{1-\theta_1}{\alpha C_E C_A} + \normLt{y_d}\Omega
                = \omega_2.
\end{equation*}
We next invoke the estimates
\eqref{eq:g_nd_diff_estimate_gamma} and
\eqref{eq:g_nd_diff_estimate_y} of \thmref{thm:g_twice_diff_u} and
\eqref{eq:g_diff_u_estimate} of \thmref{thm:g_diff_u}, as well as
the definition \eqref{eq:theta_second_order_suff} of $\theta_3$,
to arrive at
\[
    		\calJ''(\uop)h^2  \ge \del{\lambda -
                  \frac{\normSZ{v}1p\Omega}{\theta_3} }
                \normLt{h}\Omega^2.
\]
Therefore, the smallness condition \eqref{eq:v_second_order_suff} on
$v$ yields \eqref{eq:second_order_suff}. 
\end{proof}
\begin{corollary}[quadratic growth]  \label{cor:quad_growth_cond}
Under the assumptions of Theorem~\ref{thm:second_order_suff}
there exist $\theta > 0$ such that 
for all $h \in \cone \uop$ with $\normLt{h}{\I} \le \theta$ we have
\begin{align} \label{eq:J_quad_growth}
	\calJ(\uop + h) &\ge \calJ(\uop) + \frac{\lambda}{8} \normLt{h}{\I}^2 .
\end{align}
\end{corollary}
\begin{proof}
We refer to \cite[Theorem 4.23]{FTroltzsch_2010a}, which uses the 
continuity of the second-order Fr\'echet derivative (see \propref{prop:guu_Lipschitz}). 
\end{proof}

Corollary \ref{cor:quad_growth_cond} implies that there exists a unique local minimum $\uop$ solution to our OC-FBP. 
Moreover, \eqref{eq:J_quad_growth} is equivalent to 
\begin{align}  \label{eq:gradJ_quad_growth}
	\pairLt{\calJ'(u) - \calJ'(\uop), u-\uop}{\I} \ge \frac{\lambda}{4} \normLt{u - \uop}{\I}^2
	\quad \forall u \in \uop + \cone\uop . 
\end{align}

%% file: acknowledgements.tex

\section*{Acknowledgments}
This research was supported in part by NSF grants DMS-0807811 and DMS-1109325. We 
thank Michael Hinterm\"uller for pointing out reference \cite{JZowe_SKurcyusz_1979a} and
many fruitful comments on the manuscript.  
We also thank Abner Salgado for useful discussions throughout this work. 
We are especially grateful to the referees for their insightful
comments and suggestions.

%% file: main.bbl
\def\ocirc#1{\ifmmode\setbox0=\hbox{$#1$}\dimen0=\ht0 \advance\dimen0
  by1pt\rlap{\hbox to\wd0{\hss\raise\dimen0
  \hbox{\hskip.2em$\scriptscriptstyle\circ$}\hss}}#1\else {\accent"17 #1}\fi}
  \def\cprime{$'$} \def\cprime{$'$}
\begin{thebibliography}{10}

\bibitem{MElliott_MHinze_VStyles_2007a}
{\em Mini-workshop: {C}ontrol of {F}ree {B}oundaries}, Oberwolfach Rep., 4
  (2007), pp.~447--486.
\newblock Abstracts from the mini-workshop held February 11--17, 2007,
  Organized by C. M. Elliott, M. Hinze and V. Styles, Oberwolfach Reports. Vol.
  4, no. 1.

\bibitem{MElliott_YGiga_MHinze_VStyles_2010a}
{\em New directions in simulation, control and analysis for interfaces and free
  boundaries}, Oberwolfach Rep., 7 (2010), pp.~253--324.
\newblock Abstracts from the workshop held January 31--February 6, 2010,
  Organized by C. M. Elliott, Y. Giga, M. Hinze and V. Styles, Oberwolfach
  Reports. Vol. 7, no. 1.

\bibitem{RAAdams_JJFFournier_2003}
{\sc R.~A. Adams and J.~J.~F. Fournier}, {\em Sobolev spaces}, vol.~140 of Pure
  and Applied Mathematics (Amsterdam), Elsevier/Academic Press, Amsterdam,
  second~ed., 2003.

\bibitem{HAntil_RHoppe_CLinsenmann_2007a}
{\sc H.~Antil, R.H.W. Hoppe, and C.~Linsenmann}, {\em Path-following
  primal-dual interior-point methods for shape optimization of stationary flow
  problems}, J. Numer. Math., 15 (2007), pp.~81--100.

\bibitem{HAntil_RHNochetto_PSodre_2013a}
{\sc H.~Antil, R.~H. Nochetto, and P.~Sodr\'e}, {\em Optimal control of a free
  boundary problem with surface tension effects: A priori error analysis},
  submitted,  (2013).

\bibitem{MKBernauer_RHerzog}
{\sc M.~K. Bernauer and R.~Herzog}, {\em Optimal control of the classical
  two-phase {S}tefan problem in level set formulation}, SIAM J. Sci. Comput.,
  33 (2011), pp.~342--363.

\bibitem{JCDReyes_MHintermuller_2011a}
{\sc J.~C. De~Los~Reyes and M.~Hinterm{\"u}ller}, {\em A duality based
  semismooth {N}ewton framework for solving variational inequalities of the
  second kind}, Interfaces Free Bound., 13 (2011), pp.~437--462.

\bibitem{KDeckelnick_GDziuk_CElliott_2005a}
{\sc K.~Deckelnick, G.~Dziuk, and C.~M. Elliott}, {\em Computation of geometric
  partial differential equations and mean curvature flow}, Acta Numerica, 14
  (2005), pp.~139--232.

\bibitem{MDelfour_JZolesio2011}
{\sc M.~C. Delfour and J.-P. Zol{\'e}sio}, {\em Shapes and geometries}, vol.~22
  of Advances in Design and Control, Society for Industrial and Applied
  Mathematics (SIAM), Philadelphia, PA, second~ed., 2011.
\newblock Metrics, analysis, differential calculus, and optimization.

\bibitem{AErn_JLGuermond_2004a}
{\sc A.~Ern and J.~L. Guermond}, {\em Theory and practice of finite elements},
  vol.~159 of Applied Mathematical Sciences, Springer-Verlag, New York, 2004.

\bibitem{RGlowinski_1984a}
{\sc R.~Glowinski}, {\em Numerical methods for nonlinear variational problems},
  Scientific Computation, Springer-Verlag, Berlin, 2008.
\newblock Reprint of the 1984 original.

\bibitem{MHinze_SZiegenbalg_2007b}
{\sc M.~Hinze and S.~Ziegenbalg}, {\em Optimal control of the free boundary in
  a two-phase {S}tefan problem}, J. Comput. Phys., 223 (2007), pp.~657--684.

\bibitem{MHinze_SZiegenbalg_2007a}
\leavevmode\vrule height 2pt depth -1.6pt width 23pt, {\em Optimal control of
  the free boundary in a two-phase {S}tefan problem with flow driven by
  convection}, ZAMM Z. Angew. Math. Mech., 87 (2007), pp.~430--448.

\bibitem{CTKelley_1999a}
{\sc C.~T. Kelley}, {\em Iterative methods for optimization}, vol.~18 of
  Frontiers in Applied Mathematics, Society for Industrial and Applied
  Mathematics (SIAM), Philadelphia, PA, 1999.

\bibitem{TKrupenkin_JATaylor_2011a}
{\sc T.~Krupenkin and J.A. Taylor}, {\em Reverse electrowetting as a new
  approach to high-power energy harvesting}, Nat. Commun., 2 (2011), p.~448.

\bibitem{JLLions_EMagenes_1961c}
{\sc J.-L. Lions and E.~Magenes}, {\em Probl\`emes aux limites non homog\`enes.
  {IV}}, Ann. Scuola Norm. Sup. Pisa (3), 15 (1961), pp.~311--326.

\bibitem{NGMeyers_1963}
{\sc N.~G. Meyers}, {\em An {$L^{p}$}-estimate for the gradient of solutions of
  second order elliptic divergence equations}, Ann. Scuola Norm. Sup. Pisa (3),
  17 (1963), pp.~189--206.

\bibitem{SRepke_NMarheineke_RPinnau2010a}
{\sc S.~Repke, N.~Marheineke, and R.~Pinnau}, {\em Two adjoint-based
  optimization approaches for a free surface {S}tokes flow}, SIAM J. Appl.
  Math., 71 (2011), pp.~2168--2184.

\bibitem{HRoyden_1988}
{\sc H.~L. Royden}, {\em Real analysis}, Macmillan Publishing Company, New
  York, third~ed., 1988.

\bibitem{PSaavedra_RScott_1991}
{\sc P.~Saavedra and L.~R. Scott}, {\em Variational formulation of a model
  free-boundary problem}, Math. Comp., 57 (1991), pp.~451--475.

\bibitem{JSokolowski_JPZolesio_1992}
{\sc J.~Soko{\l}owski and J.~P. Zol{\'e}sio}, {\em Introduction to shape
  optimization}, vol.~16 of Springer Series in Computational Mathematics,
  Springer-Verlag, Berlin, 1992.
\newblock Shape sensitivity analysis.

\bibitem{LTartar_2007}
{\sc L.~Tartar}, {\em An introduction to {S}obolev spaces and interpolation
  spaces}, vol.~3 of Lecture Notes of the Unione Matematica Italiana, Springer,
  Berlin, 2007.

\bibitem{FTroltzsch_2010a}
{\sc F.~Tr{\"o}ltzsch}, {\em Optimal control of partial differential
  equations}, vol.~112 of Graduate Studies in Mathematics, American
  Mathematical Society, Providence, RI, 2010.
\newblock Theory, methods and applications, Translated from the 2005 German
  original by J{\"u}rgen Sprekels.

\bibitem{KVZee_EBrummelen_IAkkerman_RBorst_2010a}
{\sc KG~van~der Zee, EH~van Brummelen, I.~Akkerman, and R.~de~Borst}, {\em
  Goal-oriented error estimation and adaptivity for fluid-structure interaction
  using exact linearized adjoints}, Computer Meths. in Appl. Mech. and Eng.,
  (2010).

\bibitem{SWalker_BShapiro_RNochetto_2009a}
{\sc S.W. Walker, B.~Shapiro, and R.H. Nochetto}, {\em Electrowetting with
  contact line pinning: Computational modeling and comparisons with
  experiments}, Physics of Fluids, 21 (2009), p.~102103.

\bibitem{SWalker_ABonito_RNochetto_2010a}
{\sc S.~W. Walker, A.~Bonito, and R.~H. Nochetto}, {\em Mixed finite element
  method for electrowetting on dielectric with contact line pinning},
  Interfaces Free Bound., 12 (2010), pp.~85--119.

\bibitem{JZowe_SKurcyusz_1979a}
{\sc J.~Zowe and S.~Kurcyusz}, {\em Regularity and stability for the
  mathematical programming problem in {B}anach spaces}, Appl. Math. Optim., 5
  (1979), pp.~49--62.

\end{thebibliography}
